\documentclass[10pt]{amsart}
\usepackage{bm}
\usepackage{amsfonts}
\usepackage{amssymb,amsfonts,amsmath,amsthm,graphicx}
\usepackage{comment}
\usepackage{slashbox,mathrsfs}
\usepackage[all,poly]{xy}
\usepackage{float}

\allowdisplaybreaks

\begin{document}
\theoremstyle{plain}
\newtheorem{thm}{Theorem}[section]
\newtheorem{theorem}[thm]{Theorem}
\newtheorem*{theorem2}{Theorem}
\newtheorem{lemma}[thm]{Lemma}
\newtheorem{corollary}[thm]{Corollary}
\newtheorem{corollary*}[thm]{Corollary*}
\newtheorem{proposition}[thm]{Proposition}
\newtheorem{proposition*}[thm]{Proposition*}
\newtheorem{conjecture}[thm]{Conjecture}
\theoremstyle{definition}
\newtheorem{construction}[thm]{Construction}
\newtheorem{notations}[thm]{Notations}
\newtheorem{question}[thm]{Question}
\newtheorem{problem}[thm]{Problem}
\newtheorem{remark}[thm]{Remark}
\newtheorem{remarks}[thm]{Remarks}
\newtheorem{definition}[thm]{Definition}
\newtheorem{claim}[thm]{Claim}
\newtheorem{assumption}[thm]{Assumption}
\newtheorem{assumptions}[thm]{Assumptions}
\newtheorem{properties}[thm]{Properties}
\newtheorem{example}[thm]{Example}
\newtheorem{comments}[thm]{Comments}
\newtheorem{blank}[thm]{}
\newtheorem{observation}[thm]{Observation}
\newtheorem{defn-thm}[thm]{Definition-Theorem}

\newcommand{\sM}{{\mathcal M}}


\title[A closed formula for the asymptotic expansion of the Bergman kernel]{A closed formula for the asymptotic expansion\\ of the Bergman
kernel}

\author{Hao Xu}
        \address{Department of Mathematics, Harvard University, Cambridge, MA 02138, USA}
        \email{haoxu@math.harvard.edu}

        \begin{abstract}
        We prove a graph theoretic closed formula for coefficients in the Tian-Yau-Zelditch asymptotic expansion of the Bergman
        kernel. The formula is expressed in terms of the characteristic polynomial of the directed graphs representing Weyl
        invariants. The proof relies on a combinatorial interpretation of
        a recursive formula due to M. Engli\v s and A. Loi.
        \end{abstract}
    \maketitle

\section{Introduction}

Let $M$ be a projective algebraic manifold in $\mathbb CP^N$,
$N\geq\dim_{\mathbb C}M=n$. The hyperplane line bundle of $\mathbb
CP^N$ restricts to an ample line bundle $L$ on $M$, which is called
a polarization on $M$. A K\"ahler metric $g$ is called a polarized
metric,  if the corresponding K\"ahler form represents the first
Chern class $c_1(L)$ of $L$ in $H^2(M,\mathbb Z)$. Given any
polarized K\"ahler metric $g$, there is a Hermitian metric $h_L$ on
$L$ whose Ricci form is equal to
$$\omega_g=\frac{\sqrt{-1}}{2\pi}\sum_{i,j=1}^n g_{i\overline{j}}dz_i\wedge dz_{\overline{j}}.$$

Let $\mathcal E$ be a holomorphic vector bundle on $M$ of rank $r$
with a Hermitian metric $h_{\mathcal E}$. Then for any holomorphic
sections $U_1, U_2\in H^0(M,\mathcal E(m))$ ($\mathcal
E(m):=\mathcal E \otimes L^m$), we have the pointwise metric
$\langle U_1(x), U_2(x)\rangle_{h_{\mathcal E}\otimes h_L^m}$ and
the $L^2$-metric
\begin{equation}
(U_1,U_2)=\int_M\langle U_1(x),U_2(x)\rangle_{h_{\mathcal E}\otimes
h_L^m}\frac{\omega_g^n}{n!}.
\end{equation}

Let $S_1,\dots,S_d$ be an orthonormal basis of $H^0(M,\mathcal
E\otimes L^m)$ in the $L^2$-metric. The {\it Bergman kernel} is
defined to be the following:
\begin{equation}
B_m(x):=\sum_{j=1}^d \langle\cdot, S_j(x)\rangle S_j(x) \in {\rm
End}(\mathcal E\otimes L^m).
\end{equation}
Note that $B_m(x)$ is independent of the choice of orthonormal
basis. In particular, when $\mathcal E=\mathbb C$,
\begin{equation}
B_m(x)=\sum_{j=1}^{d}\|S_{j}(z)\|^{2}_{h_{L}^m}.
\end{equation}

The following asymptotic expansion was first proved by Zelditch
\cite{Zel} and Catlin \cite{Cat}, motivated by the convergence of
Bergman metrics started in the paper of Tian \cite{Tia} (cf. also
\cite {Bou, Rua}) following a suggestion of Yau \cite{Yau}.
\begin{theorem}{\rm\bf (Zelditch, Catlin)} \label{tyz} With the notation above, there is an asymptotic
expansion when $m\rightarrow\infty$:
\begin{equation} \label{eqb10}
B_m(x)\sim a_{0}(x)m^{n}+a_{1}(x)m^{n-1}+a_{2}(x)m^{n-2}+\cdots
\end{equation} where $a_{k}(x)\in {\rm End}(\mathcal
E\otimes L^m)$. More precisely,
$$||B_m(x)-\sum_{j=0}^k a_j(x)m^{n-j}||_{C^\mu}\leq C_{k,\mu}m^{n-k-1},$$
where $C_{k,\mu}$ depends on $k,\mu$ and the manifold $M$.
\end{theorem}

In the case of $\mathcal E=\mathbb C$, Lu \cite{Lu} computed
$a_{k}(x)$ for $k\leq3$ by using the peak section method \cite{Tia}
and proved that each $a_{k}(x)$ is a polynomial of the curvature and
its covariant derivatives. The above theorem and Lu's computation of
$a_1(x)$ have played a crucial role in Donaldson's breakthrough work
\cite{Don}. The theorem was generalized to symplectic manifolds for
Boutet de Monvel-Guillemin's almost holomorphic sections by Shiffman
and Zelditch \cite{SZ}, and Borthwick and Uribe \cite{BU}. Dai, Liu
and Ma \cite{DLM} established the full off-diagonal expansion of the
Bergman kernel on orbifolds and symplectic manifolds for the
Spin$^c$ Dirac operator by using the heat kernel method. See also
\cite{Son, RT}.

There are alternative derivations of these $a_k$ by Engli\v{s}
\cite{Eng} and Loi \cite{Loi} using the Laplace integral, by Douglas
and Klevtsov using path integral \cite{DK} when $\mathcal E=\mathbb
C$. For general $\mathcal E$, X. Wang \cite{Wan} computed $a_1$ for
the first time; L. Wang \cite{WanL} computed $a_1,a_2$; Ma and
Marinescu \cite{MM3, MM, MM4} computed $a_1,a_2$ in the symplectic
case and presented a recursive method to compute coefficients of the
expansion of more general Bergman kernels as well as the kernel of
Toeplitz operators; Berman, Berndtsson and Sj\"ostrand \cite{BBS}
gave a proof of Theorem \ref{tyz} by microlocal analysis and showed
a recursive algorithm for computing higher order terms; Liu and Lu
\cite{LL} gave a proof of Theorem \ref{tyz} using the complex
geometric method and uncovered new geometric information about the
expansion. The excellent monograph by Ma and Marinescu \cite{MM2}
contains a comprehensive introduction to the asymptotic expansion of
the Bergman kernel and its applications.

In other contexts, the asymptotic expansion of the Bergman kernel
also plays an important role in the Berezin quantization on K\"ahler
manifolds \cite{Ber}. We will discuss the Berezin transform briefly
in Section \ref{berezin}. It was applied to define the Berezin
$\star$-product on K\"ahler manifolds. Reshetikhin and Takhtajan
\cite{RTa} obtained a formula of the Berezin $\star$-product in
terms of partition functions of Feynman graphs. More explicit
graph-theoretic formulae of the Berezin and Berezin-Toeplitz $\star$-products were obtained
in \cite{Xu, Xu2}. We remark that Kontsevich's celebrated formula
\cite{Kon} for a $\star$-product on Poisson manifolds was also
written as a summation over graphs.

In order to state the main results in this paper, we assume
$\mathcal E=\mathbb C$ and introduce some terminologies in graph
theory.

A {\it digraph} (directed graph) $G=(V,E)$ is defined to be a finite
set $V$ (whose elements are called vertices) and a multiset $E$ of
ordered pairs of vertices, called directed edges. Throughout the
paper, we allow a digraph to have loops and multi-edges. The
adjacency matrix $A=A(G)$ of a digraph $G$ with $n$ vertices is a
square matrix of order $n$ whose entry $A_{ij}$ is the number of
directed edges from vertex $i$ to vertex $j$. The {\it outdegree}
$\deg^+(v)$ and {\it indgree} $\deg^-(v)$ of a vertex $v$ are
defined to be the number of outward and inward edges at $v$
respectively.

A digraph $G$ is called {\it connected} if the underlying undirected
graph is connected, and {\it strongly connected} if there is a
directed path from each vertex in $G$ to every other vertex.

We call a directed graph $G=(V,E)$ {\it stable} if at each vertex
$v$ both the outdegree  and indegree  are no less than $2$. The set
of stable graphs will be denoted by $\mathcal G$. The {\it weight}
of $G$ is defined to be $|E|-|V|$. We will define a natural function
$z(G)$ on $\mathcal G$ such that the coefficients $a_k$ can be
written as a sum over $\mathcal G(k)$, the set of stable graphs of
weight $k$ (see Section \ref{secgraph} for details)
\begin{equation} \label{eqc}
a_k(x)=\sum_{G\in\mathcal G(k)} z(G)\cdot G, \quad z(G)\in\mathbb Q.
\end{equation}
So we may regard $z$ as a map from the set $\mathcal G$ of all
stable graphs to $\mathbb Q$, from which we can easily recover the
curvature-tensor expression of these $a_k$'s (see Example
\ref{tyza2} where $a_1$ and $a_2$ are computed, and Appendix
\ref{apthree} where $a_3$ is computed).

The main result of this paper is the following theorem.
\begin{theorem} \label{main} Let $G=(V,E)\in\mathcal G$ be a stable graph with the adjacency matrix $A$.
\begin{enumerate}
\item[i)] If $G$ is  a disjoint
union of connected subgraphs $G=G_1\cup\dots\cup G_n$, then we have
\begin{equation}
z(G)=\prod_{j=1}^n z(G_j)/|Sym(G_1,\dots,G_n)|,
\end{equation}
where $Sym(G_1,\dots,G_n)$ denote the permutation group of these $n$
connected subgraphs.

\item[ii)] If $G$ is connected but not strongly connected, then
\begin{equation}
 z(G) = 0.
\end{equation}

\item[iii)] If $G$ is strongly
connected, then
\begin{equation}
 z(G) = -\frac{\det(A-I)}{|{\rm Aut}(G)|},
\end{equation}
where $I$ is the identity matrix.
\end{enumerate}

\end{theorem}

We computed $z(G)$ for $G\in\mathcal G(k),\, k\leq4$ using Loi's
recursive formula \eqref{eqloi}. They match with values computed by
the above theorem (see the appendix).

As pointed out by the referee, it would be interesting to see the
implication of the above theorem on the relation between the Bergman
kernel on K\"ahler manifolds and the path integral (cf. \cite{DK}).
We hope our work will find application in Fefferman's program of
studying the Bergman kernel of a strong pseudoconvex domain as an
analogy of the heat kernel of Riemannian manifolds (cf. \cite{Ale,
Eng, Fef, Hir, Xu}).

We also obtained some interesting combinatorial identities, for
example we proved the following formula for Bernoulli numbers $B_k$
in corollary \ref{bern}.
\begin{equation} \label{eqbern}
B_k=(-1)^{k} k\sum_{G\in\mathcal G(k)}\frac{\det(A-I)}{|{\rm
Aut}(G)|}\cdot\epsilon(G)\prod_{v\in V}(\deg^+(v)-1)! ,\quad k\geq1,
\end{equation}
where $\epsilon(G)$ is the number of Euler tours in $G$. Note that
in the right-hand side, $\epsilon(G)=0$ unless $G$ is connected and
balanced, and hence strongly connected. A digraph is called {\it balanced} if
$\deg^+(v)=\deg^-(v)$ for each vertex $v$.

\

\noindent{\bf Acknowledgements} The author is grateful to Professor
Kefeng Liu for his kind help over many years. The author benefited a
lot from Professor Shing-Tung Yau's seminars. The author thanks
Professor Xiaonan Ma for helpful comments. The author thanks the
referee for very helpful suggestions which greatly improved the
presentation of the paper.

 \vskip 30pt

 \section{Tensor calculus on K\"ahler manifolds} \label{sectensor}
 Let $(M,g)$ be a K\"ahler manifold of dimension
$n$. Locally the K\"ahler form is given by
$$\omega_g=\frac{\sqrt{-1}}{2\pi}\sum_{i,j=1}^n
g_{i\overline{j}}dz_i\wedge dz_{\overline{j}}.$$

We will use the Einstein summation convention. The indices
$i,j,k,\dots$ run from $1$ to $n$, while Greek indices
$\alpha,\beta,\gamma$ may represent either $i$ or $\bar i$. Let
$\det g$ be the determinant of the Hermitian matrix $(g_{i\bar j})$
and $(g^{i\bar j})$ be the inverse of the matrix $(g_{i\bar j})$. We
also use the notation
\begin{equation}
g_{j\bar
k\alpha_1\alpha_2\dots\alpha_m}:=\partial_{\alpha_1\alpha_2\dots\alpha_m}g_{j\bar
k}.
\end{equation}

The curvature tensor is defined as
\begin{equation}\label{eqcur1}
R_{i\bar jk\bar l} =-g_{i\bar j k\bar l}+g^{m\bar p}g_{m\bar j\bar
l}g_{i\bar p k}.
\end{equation}
The Ricci tensor is
\begin{equation}\label{eqcur2}
R_{i\bar j}=g^{k\bar l}R_{i\bar jk\bar l}=-\partial_i\partial_{\bar
j}\log(\det g)
\end{equation}
and the scalar curvature is the trace of the Ricci curvature
\begin{equation}\label{eqcur3}
\rho=g^{i\bar j}R_{i\bar j}.
\end{equation}

The covariant derivative of a tensor field
$T_{\beta_1\dots\beta_q}^{\alpha_1\dots\alpha_p}$ is defined by
\begin{equation}\label{eqcur4}
T_{\beta_1\dots\beta_q;\gamma}^{\alpha_1\dots\alpha_p}=\partial_{\gamma}T_{\beta_1\dots\beta_q}^{\alpha_1\dots\alpha_p}-\sum_{i=1}^q
\Gamma_{\gamma\beta_i}^{\delta}T_{\beta_1\dots\beta_{i-1}\delta\beta_{i+1}\dots\beta_q}^{\alpha_1\dots\alpha_p}
+\sum_{j=1}^p\Gamma_{\delta\gamma}^{\alpha_j}T_{\beta_1\dots\beta_q}^{\alpha_1\dots\alpha_{j-1}\delta\alpha_{j+1}\dots\alpha_p},
\end{equation}
where the Christoffel symbols $\Gamma_{\beta\gamma}^\alpha=0$ except
for
\begin{equation}
\Gamma_{jk}^i=g^{i\bar l}g_{j\bar l k},\quad \Gamma_{\bar j\bar
k}^{\bar i}=g^{l\bar i}g_{l\bar j\bar k}.
\end{equation}

\begin{lemma} For tensors in K\"ahler geometry, we have the following identities:
\begin{enumerate}
\item[i)] The K\"ahler metric $g$ satisify
\begin{equation} \label{eqcur6}
\partial_i g_{j\bar k}=\partial_j g_{i\bar k},\quad \partial_{\bar
l}g_{j\bar k}=\partial_{\bar k} g_{j\bar l}.
\end{equation}
\item[ii)] The derivative of $g^{m\bar l}$ satisfy
\begin{equation}\label{eqcur7}
\partial_{\alpha}g^{m\bar l}=-g^{p\bar l}g^{m\bar q}g_{p\bar
q\alpha}.
\end{equation}
\item[iii)] The derivative of $\det g$ satisfy
\begin{equation}\label{eqdet}
\partial_{\alpha}\det g=\det g\cdot g^{m\bar l}g_{m\bar
l\alpha}.
\end{equation}
\item[iv)] The curvature tensor satisfy the first Bianchi identity
\begin{equation}
R_{i\bar jk\bar l}=R_{i\bar l k\bar j}=R_{k\bar j i\bar l}.
\end{equation}
\item[v)] The covariant derivatives of the curvature tensor
satisfy the second Bianchi identity
\begin{equation}
R_{i\bar j k\bar l;m}=R_{m\bar j k\bar l;i}=R_{i\bar j m\bar
l;k},\quad R_{i\bar j k\bar l;\bar p}=R_{i\bar p k\bar l;\bar
j}=R_{i\bar j k\bar p;\bar l}.
\end{equation}
\item[vi)] The Ricci formula gives the difference when we
interchange two covariant derivative indices
\begin{equation}\label{eqcur5}
T_{\beta_1\dots\beta_q;i\bar
j}^{\alpha_1\dots\alpha_p}-T_{\beta_1\dots\beta_q;\bar j
i}^{\alpha_1\dots\alpha_p}=\sum_{k=1}^q R_{\beta_k i\bar
j}^{\gamma}T_{\beta_1\dots\beta_{k-1}\gamma\beta_{k+1}\dots\beta_q}^{\alpha_1\dots\alpha_p}
-\sum_{l=1}^p R_{\gamma i\bar
j}^{\alpha_l}T_{\beta_1\dots\beta_q}^{\alpha_1\dots\alpha_{l-1}\gamma\alpha_{l+1}\dots\alpha_p},
\end{equation}
where $R_{\bar l i\bar j}^{\bar k}= -g^{m\bar k}R_{m\bar l i\bar j}$,
$R_{l i\bar j}^{k}= g^{k\bar m}R_{l\bar m i\bar j}$ and $R_{\bar l
i\bar j}^{k}= R_{l i\bar j}^{\bar k}=0$.
\end{enumerate}
\end{lemma}

Recall that around each point $x$ on a K\"ahler manifold, there
exists a normal coordinate such that
\begin{equation} \label{eqnormal}
g_{i\bar j}(x)=\delta_{ij}, \qquad g_{i\bar j k_1\dots
k_r}(x)=g_{i\bar j \bar l_1\dots\bar l_r}(x)=0
\end{equation}
for all $r\leq N\in \mathbb N$, where $N$ can be chosen arbitrary
large.

\begin{remark} \label{rm1}
Fix a normal coordinate around $x$ on a K\"ahler manifold. By
\eqref{eqcur1} and \eqref{eqcur4}, it is not difficult to see
inductively that any covariant derivative $R_{i\bar jk \bar l;
\alpha_1\dots \alpha_r}$ is canonically equal to a polynomial of
$g_{a\bar b \alpha}$ and their partial derivatives. By
\eqref{eqnormal}, any monomial containing $g_{i\bar j k_1\dots k_r}$
or $g_{i\bar j \bar l_1\dots\bar l_r}$ vanishes at $x$. For example
$$R_{i\bar jk \bar l}(x)=-g_{i\bar jk \bar l}(x),$$
$$R_{i\bar jk \bar l; p\bar q}(x)=-g_{i\bar jk \bar l p\bar q}(x)
+g^{s\bar t}(g_{p\bar j s\bar l}g_{i\bar q k\bar t}+g_{k\bar j s\bar
l}g_{p\bar q i\bar t}+g_{i\bar j s\bar l}g_{k\bar q p\bar t})(x),$$
$$R_{i\bar jk \bar l; p_1\dots p_r}(x)=-g_{i\bar jk \bar l
p_1\dots p_r}(x),\quad \forall\, r\geq 1,$$
$$R_{i\bar jk \bar l; \bar p_1\dots \bar
p_r}(x)=-g_{i\bar jk \bar l \bar p_1\dots \bar p_r}(x),\quad
\forall\, r\geq 1.$$

Note that these identities hold only at the point $x$. However they
can easily recover the original identities (hold in a neighborhood
of $x$), since from \eqref{eqcur1} and \eqref{eqcur4}, we see that
$g_{i\bar j k\bar l \beta_1\dots\beta_{r}}$ can be inductively
expressed as a canonical polynomial of covariant derivatives of
curvature tensors, denoted by $D(g_{i\bar j k\bar l
\beta_1\dots\beta_{r}})$. For example
$$D(g_{i\bar j k\bar l})=-R_{i\bar j k\bar l},$$
$$D(g_{i\bar j k\bar l p\bar q})=-R_{i\bar jk \bar l; p\bar q}+g^{s\bar t}(R_{p\bar j s\bar l}R_{i\bar q k\bar t}+R_{k\bar j s\bar
l}R_{p\bar q i\bar t}+R_{i\bar j s\bar l}R_{k\bar q p\bar t}).$$ In
general, we can inductively get
\begin{equation}
D(g_{i\bar j k\bar l \beta_1\dots\beta_{r}})=-R_{i\bar j k\bar l;
\beta_1\dots\beta_{r}}+\text{ covariant derivatives of lower order}.
\end{equation}

In Section \ref{secgraph}, this observation will allow us to store
Weyl invariants as polynomials of $g_{a\bar b c \bar d}$ and their
derivatives. The advantage is that we do not need to deal with the
problem of exchanging indices as in the Ricci formula
\eqref{eqcur5}.
\end{remark}

We will now define the concept of an admissible tree, which is a
directed tree with half-edges. We call the two trees in Figure
\ref{fig1} {\it primitive admissible trees}.
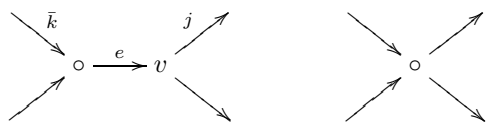
\begin{figure}[h]
$\xymatrix@C=7mm@R=5mm{\ar[dr]^{\bar k} && &\\ &\circ \ar[r]^{e} & v
\ar[ur]^j \ar[dr] &
\\\ar[ur] &&&}$
\qquad\quad $\xymatrix@C=7mm@R=5mm{ \ar[dr] & &\\  &\circ \ar[ur]
\ar[dr]&
\\  \ar[ur] & &}$
\caption{The atomic admissible tree and a trivial admissible tree}
\label{fig1}
\end{figure}

We can define three types of actions by an outward half-edge $i$ or
an inward half-edge $\bar i$ on a directed tree with half-edges. Let
us take the left-hand primitive admissible tree as an example:
\begin{enumerate}
\item The action on a vertex $v$ is to add the half edge to the
vertex (see Figure \ref{fig5}).
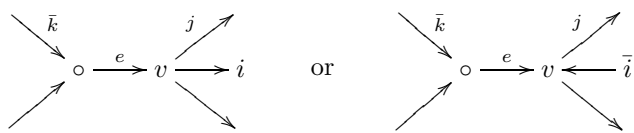
\begin{figure}[h]
\begin{tabular}{c}$\xymatrix@C=7mm@R=5mm{\ar[dr]^{\bar k} && &\\ &\circ
\ar[r]^{e} & v \ar[r]  \ar[ur]^j \ar[dr] & i
\\\ar[ur] &&&}$ \end{tabular} \quad\text{ or
}\quad \begin{tabular}{c}$\xymatrix@C=7mm@R=5mm{\ar[dr]^{\bar k} && &\\
&\circ \ar[r]^{e} & v \ar[ur]^j \ar[dr] & {\bar i} \ar[l]
\\\ar[ur] &&&}$ \end{tabular}
\caption{Action on a vertex $v$} \label{fig5}
\end{figure}
\item The action on a half-edge with the same direction (i.e. $i$ can only act on outward half-edges and $\bar i$ only act on inward half-edges) is to
put them together on a new vertex (see Figure \ref{fig4}).
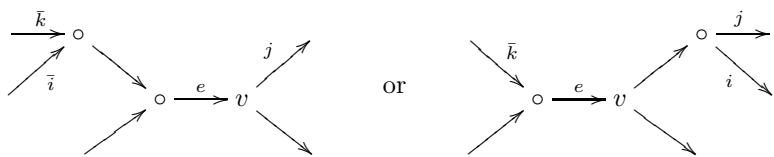
\begin{figure}[h]
\begin{tabular}{c}$\xymatrix@C=7mm@R=5mm{  \ar[r]^{\bar k} &\circ \ar[dr] & & &\\
\ar[ur]_{\bar i} & &\circ \ar[r]^e & v \ar[ur]^j \ar[dr]&
\\  & \ar[ur] & & &}$ \end{tabular}
\quad\text{ or }\quad
\begin{tabular}{c}$\xymatrix@C=7mm@R=5mm{ \ar[dr]^{\bar k} & & &\circ \ar[r]^j \ar[dr]_i &\\  &\circ \ar[r]^e & v \ar[ur] \ar[dr] & &
\\ \ar[ur] & & & &}$ \end{tabular}
\caption{Action on a half-edge} \label{fig4}
\end{figure}
\item The action on an edge $e$ is to generate a new vertex
to split $e$ and add the half-edge to the new vertex (see Figure
\ref{fig3}).
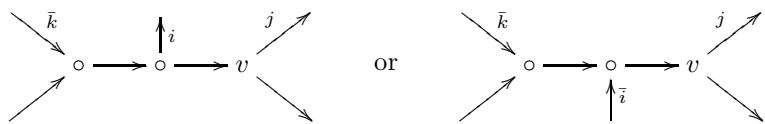
\begin{figure}[h]
\begin{tabular}{c}$\xymatrix@C=7mm@R=5mm{ \ar[dr]^{\bar k} & & & &\\  & \circ \ar[r] &\circ \ar[r] \ar[u]_i& v \ar[ur]^j \ar[dr]&
\\  \ar[ur]& & & &}$ \end{tabular}\quad\text{ or
}\quad
\begin{tabular}{c}$\xymatrix@C=7mm@R=5mm{ \ar[dr]^{\bar k} & & & &\\  & \circ \ar[r] &\circ \ar[r]& v \ar[ur]^j \ar[dr]&
\\ \ar[ur]& &\ar[u]_{\bar i} & &}$ \end{tabular}
\caption{Action on an edge $e$} \label{fig3}
\end{figure}
\end{enumerate}

For any two trees $T_1$ and $T_2$, we can join them by gluing an
outward half-edge of $T_1$ to an inward half-edge of $T_2$, or vice
versa.
\begin{definition}
An {\it admissible tree} (with half-edges) is defined to be the join
of finite number of trees which can be obtained by a finite number
of the above three actions on the primitive admissible trees in
Figure \ref{fig1}. For an admissible tree, we usually label its
outward half-edges and inward half-edges with distinct indices
$\{i,j,k,\dots\}$ and $\{\bar l,\bar m,\bar p,\dots\}$ respectively.
\end{definition}

An admissible tree is called {\it decomposable} if we can cut an
edge to get two admissible trees (see Figure \ref{fig2}). Otherwise,
it is called {\it indecomposable.}

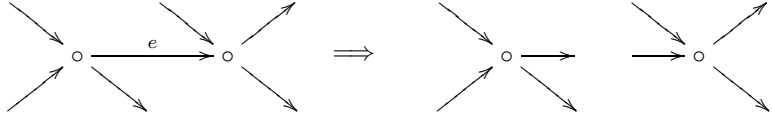
\begin{figure}[h]
\begin{tabular}{c}$\xymatrix@C=7mm@R=5mm{& \ar[dr] & & \ar[dr] & &\\ & & \circ \ar[rr]^e \ar[dr]& & \circ \ar[ur] \ar[dr]&
\\ & \ar[ur]& & & &}$ \end{tabular} $\Longrightarrow$ \quad
\begin{tabular}{c}$\xymatrix@C=7mm@R=5mm{ \ar[dr] & &\\  &\circ \ar[r] \ar[dr]&
\\  \ar[ur] & &} \quad \xymatrix@C=7mm@R=5mm{ \ar[dr] & &\\  \ar[r] &\circ \ar[ur] \ar[dr]&
\\  & &}$ \end{tabular}
\caption{Decomposition by cutting an edge $e$} \label{fig2}
\end{figure}

\begin{definition} \label{defD}
For any $g_{i\bar j k\bar l \beta_1\dots\beta_{r}}$, as in remark
\ref{rm1}, there exists a canonical polynomial of covariant
derivatives of curvature tensor, denoted by $D(g_{i\bar j k\bar l
\beta_1\dots\beta_{r}})$, that coincides with $g_{i\bar j k\bar l
\beta_1\dots\beta_{r}}$ at the center of the normal coordinate. By
\eqref{eqcur1} and \eqref{eqcur4}, $D(g_{i\bar j k\bar l
\beta_1\dots\beta_{r}})$ may be written as a unique universal
polynomial of $g_{a\bar b\gamma_{1}\dots\gamma_t}$ with $t\geq 1$.
\end{definition}

For example,
$$D(g_{i\bar j k\bar l})=g_{i\bar j k\bar l}-g^{m\bar p}g_{m\bar j\bar l}g_{ik\bar p}.$$
Obviously, the above identity can be expressed in terms of
admissible trees
\begin{equation}
D\left(\begin{tabular}{c} \xymatrix@C=5mm@R=5mm{\ar[dr]^{\bar l} &&
\\ &\circ \ar[ur]^k \ar[dr]^i &
\\\ar[ur]^{\bar j} &&}
\end{tabular}\right) = \begin{tabular}{c}\xymatrix@C=5mm@R=5mm{\ar[dr]^{\bar l} &&
\\ &\circ \ar[ur]^k \ar[dr]^i &
\\\ar[ur]^{\bar j} &&}\end{tabular} - \begin{tabular}{c}\xymatrix@C=7mm@R=5mm{\ar[dr]^{\bar l} && &\\ &\circ \ar[r] & \circ \ar[ur]^k \ar[dr]^i &
\\\ar[ur]^{\bar j} &&&}\end{tabular}
\end{equation}

The map $D$ can be naturally extended to be defined on any
admissible tree with both indegree and outdegree $\geq2$ at each
vertex. We simply apply $D$ to each vertex while keeping index
pairings and expanding linearly.

Let $\mathcal T_m(\alpha_1,\dots,\alpha_r)$ denote the set of all
indecomposable admissible trees with $m$ vertices and $r$ half edges
labeled by the set of indices $\{\alpha_1,\dots,\alpha_r\}$.

\begin{lemma} \label{covar} We have
\begin{equation} \label{eqcur9}
D(g_{i\bar j k\bar l \beta_1\dots\beta_{r}})=\sum_{m=1}^{r+2}\
\sum_{T\in \mathcal T_m(i,\bar j, k,\bar l, \beta_1,\dots,\beta_r)}
(-1)^{m+1} T.
\end{equation}
\end{lemma}
\begin{proof} (Sketch) We first note that equation \eqref{eqcur7}
corresponds to the action on edges in Figure \ref{fig3} and the
multiplication by $\Gamma$ in equation \eqref{eqcur4} corresponds to
the action on half-edges in Figure \ref{fig4}.

By \eqref{eqcur1}, we have
\begin{equation}
R_{i\bar jk\bar l \beta_1\dots\beta_{r}} =-g_{i\bar j k\bar l
\beta_1\dots\beta_{r}}+(g^{m\bar p}g_{m\bar j\bar l}g_{ik\bar
p})_{\beta_1\dots\beta_{r}}
\end{equation}
The left-hand side is a derivative of $R_{i\bar jk\bar l}$. By
induction, we can see that $D((g^{m\bar p}g_{m\bar j\bar l}g_{ik\bar
p})_{\beta_1\dots\beta_{r}})$ is a sum over all decomposable
admissible trees with the set of indices $\{i,\bar j, k,\bar l,
\beta_1,\dots,\beta_r\}$ such that no irreducible component of the
tree can contain indices from both of the two sets $\{i,k\}$ and
$\{\bar j,\bar l\}$. We can apply \eqref{eqcur4} inductively to see
that all trees from $R_{i\bar jk\bar l \beta_1\dots\beta_{r}}$ are
admissible and get the desired equation \eqref{eqcur9}.
\end{proof}

\vskip 30pt
\section{The local Bergman kernel and the Berezin transform}
\label{berezin}

In this section, we consider only the case that $\mathcal E=\mathbb
C$. So the Bergman kernel $B_m(x)$ is a scalar function on $M$.

Now we introduce the local Bergman kernel following \cite{Eng}. Let
$\Omega$ be a bounded domain in $\mathbb C^n$ and $\Phi$ be a
K\"{a}hler potential for a K\"{a}hler metric $g$ on $\Omega$
satisfying
$$\omega_g=\frac{\sqrt{-1}}{2\pi}\partial\overline\partial\Phi.$$
To simplify, we may assume $\Omega$ is equipped with a normal
coordinate.

Let $\Phi(x,y)$ be an almost analytic extension to a neighborhood of
the diagonal, i.e. $\bar\partial_x\Phi$ and $\partial_y\Phi$ vanish
to infinite order for $x=y$ (cf. \cite{BS}). We can assume
$\overline{\Phi(x,y)}=\Phi(y,x)$. Consider the real valued function
$$D(x, y) = \Phi(x, x) + \Phi(y, y)-\Phi(x, y)-\Phi(y, x),$$
which is called the Calabi diastatic function \cite{Cal}. It is
easily seen that the function $D(x, y)\geq0$ and $D(x, y)=0$ if and
only $x=y$.

We need the following important result of Engli\v{s} \cite{Eng}.
\begin{theorem}{\rm(Engli\v{s})} \label{eng}
There is an asymptotic expansion for the Laplace integral
$$\int_{\Omega} f(y)e^{-mD(x,y)}\frac{\omega_g^n(y)}{n!} \sim \frac{1}{m^n}\sum_{j\geq0}m^{-j}R_j(f)(x),$$
where $R_j: C^\infty(\Omega)\rightarrow C^\infty(\Omega)$ are
explicit differential operators defined by
\begin{equation} \label{eqeng}
R_j f(x)=\frac{1}{\det g}\sum_{k=j}^{3j}\frac{1}{k!(k-j)!}L^k(f\det
g S^{k-j})|_{y=x},
\end{equation}
where $L$ is the (constant-coefficient) differential operator
$$L f(y)=g^{i\bar j}(x)\partial_i\partial_{\bar j} f(y)$$
and the function $S(x,y)$ satisfies $$S=\partial_\alpha
S=\partial_{\alpha\beta}S=\partial_{i_1 i_2\dots
i_m}S=\partial_{\bar i_1\bar i_2\dots \bar i_m}S=0 \quad \text{ at }
y=x,$$
$$\partial_{i\bar j\alpha_1\alpha_2\dots
\alpha_m}S|_{y=x}=-\partial_{\alpha_1\alpha_2\dots \alpha_m}g_{i\bar
j}(x).$$ In particular, we have
$$\begin{cases}
R_0=id
\\ R_1(f)=\Delta f-\frac12 f\rho.
\\ R_2(f)=\frac12\Delta^2f-\frac12 L_{Ric}f-\frac\rho
2\Delta f-\frac12(\rho_i f_{;i}+\rho_{\bar i}f_{;\bar i})\\\qquad
\qquad -(\frac13\Delta\rho-\frac18\rho^2-\frac16
|Ric|^2+\frac1{24}|R|^2)f.
\end{cases}$$
\end{theorem}
Here $L_{Ric}f=R_{j\bar i}f_{;i\bar j},\, |Ric|^2=R_{i\bar
j}R_{j\bar i},\, |R|^2=R_{i\bar j k\bar l}R_{j\bar i l\bar k}$.
Note that our convention of curvatures $R_{i\bar j k\bar l},
R_{i\bar j}, \rho$ all differ by a minus sign with that of
\cite{Eng}.

We need the following extension of Tian-Yau-Zelditch asymptotic
expansion \cite{Cat, Eng2, KS}.
\begin{theorem} \label{tyz2}
Let $M$ be a compact complex manifold endowed with a polarized
K\"{a}hler metric $g$. Let $B_m(x, y)$ denote an almost analytic
extension of $B_m(x)$ to an open neighborhood, say $U \times U$ of
the diagonal. Then, for $U$ sufficiently small, $B_m(x, y)$ admits
an asymptotic expansion (as $m\rightarrow +\infty$) of the form:
$$B_m(x, y)\sim\sum_{j\geq0} a_j(x,y)m^{n-j}.$$
\end{theorem}

For $\alpha > 0$, consider the weighted Bergman space of all
holomorphic function on $\Omega$ square-integrable with respect to
the measure $e^{-\alpha \Phi}\frac{w_g^n}{n!}$.

We denote by $K_\alpha(x,y)$ the reproducing kernel. As pointed out
by Engli\v s \cite{Eng}, it is often the case  that the following
holds. (e.g. if $(\Omega, g_{i\bar j})$ is a bounded symmetric with
the invariant metric or whenever $\Omega$ is strongly pseudoconvex
with real analytic boundary.)
\begin{itemize}
\item[(1)] $K_\alpha(x,y)$ has an asymptotic expansion in a small
neighborhood of the diagonal
\begin{equation}\label{eqb1}
K_\alpha (x,y)\sim e^{\alpha \Phi (x,y)} \sum^\infty_{k=0} b_k (x,y)
\alpha^{n-k}.
\end{equation}
\item[(2)] For any neighborhood $U$ of a point $x \in \Omega$ and a
bounded measurable function $f$,
\begin{equation}\label{eqb2}
\int_{\Omega\backslash U} f(y)\frac{|K_\alpha
(x,y)|^2}{K_\alpha(x,x)} e^{-\alpha \Phi (y)}
\frac{w_g^n(y)}{n!}=o(\alpha^{-k}), \quad \forall\,
k\geq1.\end{equation}
\end{itemize}

Denote $b_k(x,x)$ by $b_k(x)$.
\begin{theorem} \label{berg}
For any $k\geq0$, we have $a_k(x)=b_k(x)$.
\end{theorem}
This theorem means that the global Bergman kernel can be
approximated by the local Bergman kernel. An analytic proof of this
theorem can be found in \cite{BBS}. Here we present a simple proof
which is implicit in \cite{Loi}.

The {\it Berezin transform} is the integral operator
\begin{equation}
I_\alpha f(x)=\int_\Omega
f(y)\frac{|K_\alpha(x,y)|^2}{K_\alpha(x,x)}e^{-\alpha
\Phi(y)}\frac{w_g^n(y)}{n!}.
\end{equation}
At any point for which $K_\alpha (x,x)$ invertible, the integral
converges for each bounded measurable function $f$ on $\Omega$.

Assume that \eqref{eqb1} and \eqref{eqb2} hold, then the Berezin
transform has an asymptotic expansion (see \cite{Eng})
\begin{equation}
I_\alpha f(x)=\sum^\infty_{k=0} Q_k f(x)\alpha^{-k},
\end{equation}
where $Q_k$ are differential operators.

We have that for $(x,y)$ near the diagonal,
\begin{equation} \label{eqb12}
\frac{|K_\alpha (x,y)|^2}{K_\alpha(x,x)} e^{-\alpha \Phi
(y)}=e^{-\alpha D(x,y)} \sum^\infty_{k=0} \tilde b_{k} (x,y)
\alpha^{n-k},
\end{equation}
where $\tilde b_0=1, \tilde b_1(x,y)=b_1(x,y)+b_1(y,x)-b_1(x,x),$
etc.

Fix a bounded neighborhood $U$ of $x$ such that \eqref{eqb12} holds.
Applying Theorem \ref{eng}, we get
\begin{equation} \label{eqb3}
Q_k f(x)=\sum^k_{j=0} R_j (\tilde b_{k-j}(x,y)f(y))|_{y=x},
\end{equation}
where the operators $R_j$ apply to the $y$-variable.

Since $Q_k f=0$ when $k>0$ and $f$ is analytic or anti-analytic,
setting $f=1$ in \eqref{eqb3}, we get
\begin{equation}\label{eqb4}
b_k(x)=-\sum^k_{j=1} R_j(\tilde b_{k-j}(x,y))|_{y=x}, \quad k\geq 1.
\end{equation}

The following result from quantization of K\"ahler manifolds can be
found in \cite{CGR, Loi}.

\begin{theorem} \label{qua}
In the same hypothesis of Theorem \ref{tyz2}, we have a globally
defined function on $M$. $$\psi_m
(x,y)=\frac{e^{-mD(x,y)}|B_m(x,y)|^2}{B_m (x) B_m(y)}.$$ Then
\begin{equation}\label{eqb5}
\int_M \psi_m(x,y) B_m(y)\frac{w_g^n(y)}{n!}=1
\end{equation}
Moreover, for any neighborhood $U$ of $x\in M$ and every smooth
function $f$ on $M$.
$$\int_{M\backslash U}\psi_m(x,y) B_m(y)\frac{w_g^n
(y)}{n!}=o(m^{-k}), \quad \forall\, k\geq1.$$
\end{theorem}

From \eqref{eqb5} and Theorem \ref{eng}, we get
$$\sum_{j=0}^k R_j(\tilde a_{k-j}(x,y))|_{y=x}=0,$$ where $\tilde
a_0=1, \tilde a_1(x,y)=a_1(x,y)+a_1(y,x)-a_1(x)$, etc. Note that
this is exactly the same recursive equation as \eqref{eqb4}. Since
$a_0(x)=b_0(x)=1$, we must have $a_k(x)=b_k(x)$ for all $k\geq0$. So
we conclude the proof of Theorem \ref{berg}.

\vskip 30pt
\section{The Bergman kernel of vector bundles}

Let $h(x)$ be a Hermitian metric on $\Omega\times\mathbb C^r$. For
$\alpha
> 0$, consider the weighted Bergman space of all holomorphic
function $f: \Omega\rightarrow \mathbb C^r$ square-integrable with
respect to the norm
\begin{equation}
||f||^2=\int_{\Omega}\langle f,f\rangle_{h} e^{-\alpha
\Phi}\frac{w_g^n}{n!}.
\end{equation}

Now the Bergman kernel $K_\alpha(x,y)$ takes value in matrices. The
results in the last section could be extended without difficulty to
the matrix case.

By the property of reproducing kernels, we have for $v\in \mathbb
C^r$,

\begin{equation} \label{eqb6}
\langle K_\alpha(x,x)f(x), v\rangle=\int_\Omega\langle
K_\alpha(x,y)f(y), K_\alpha(x,y)v\rangle e^{-\alpha\phi(y)}g(y)dy.
\end{equation}

We take $f(x)$ to be a constant function and take the inner product
to be the Hermitian metric $h_{\alpha\bar \beta}$. Let
$H=h_\alpha^\beta \in \Gamma(End(\mathcal {E}))$ and apply Theorem
\ref{eng} to \eqref{eqb6}, we get
\begin{equation} \label{eqb7}
a_k(x)=\sum_{\ell+i+j=k}R_{\ell}(a_i(x,y)a_j(y,x)H)|_{y=x}.
\end{equation}

\begin{remark}
The connection of Engli\v{s}' work and the asymptotic expansion of
Tian-Yau-Zelditch was studied by Loi \cite{Loi}, who obtained the
formula \eqref{eqb7} when $\mathcal E=\mathbb C$ using Theorem
\ref{qua}.
\end{remark}

From the curvature equation $$\Theta_{i\bar j}H=-H_{i\bar
j}+\partial_i H\partial_{\bar j}H,$$ we have
\begin{align}
(H_{i\bar i})_{/j\bar
j}&=-\Theta_{i\bar i / j\bar j}H-\Theta_{i\bar i}H_{j\bar j}+H_{i\bar j}H_{j\bar i} \nonumber\\
\label{eqcur} &=-\wedge \partial \bar \partial \wedge \Theta+\wedge
\Theta \wedge\Theta+\Theta_{i\bar j}\Theta_{j\bar i}.
\end{align}

Since $R_0$ is the identity operator, we have
\begin{equation} \label{eqb8}
a_k(x)=-R_k(H)-\sum_{\substack{i+j=k
\\i,j\geq1}}a_i(x)a_j(x)H-\sum_{\substack{\ell+i+j=k\\1\leq\ell\leq
k-1}}R_\ell(a_i(x,y)a_j(y,x)H)|_{y=x}.
\end{equation}
Then it is not difficult to compute the first few terms
\begin{equation}\label{eqtyz1} a_1(x)=-\Delta H+\frac{\rho}{2}H=\frac{\rho}{2}I+\wedge\Theta.
\end{equation}
Note that in the second equation, we used $H(0)=I$.

For $k=2$, we need a little more work.
$$a_2(x)=-R_2(H)-a_1(x)^2H-R_1(a_1(x,y)H)|_{y=x}-R_1(a_1(y,x)H)|_{y=x}.$$
We compute terms in the right-hand side one by one. By formula
\eqref{eqcur}, we have
\begin{align*}
R_2(H)=&\frac{1}{2}\Delta\Delta H-\frac{1}{2}R_{i\bar j}H_{j\bar i}-\frac{\rho}{2}\Delta H-(\frac{1}{3}\Delta\rho-\frac{1}{8}\rho^2-\frac{1}{6}|Ric|^2+\frac{1}{24}|R|^2)H\\
=&(-\frac{1}{2}\wedge\partial\bar\partial\wedge\Theta+\frac{1}{2}\Theta_{i\bar
j}\Theta_{j\bar
i}+\frac{1}{2}\wedge\Theta\wedge\Theta)+\frac{1}{2}R_{i\bar
j}\Theta_{j\bar
i}+\frac{\rho}{2}\wedge\Theta\\
&-(\frac{1}{3}\Delta\rho-\frac{1}{8}\rho^2-\frac{1}{6}|Ric|^2+\frac{1}{24}|R|^2)I.
\end{align*}
Substituting $a_1(x)$ computed in \eqref{eqtyz1}, we have
$$a_1(x)^2H=\frac{\rho^2}{4}I+\rho\wedge\Theta+\wedge\Theta\wedge\Theta.$$

From almost-analyticity of $a_1(x,y)$, we get
\begin{align*} R_1(a_1(x,y)H)|_{y=x}&=\Delta(a_1(x,y)H)|_{y=x}-\frac{\rho}{2}a_1(x)H\\
&=a_1(x)\Delta
H-\frac{\rho}{2}a_1(x)\\&=-\frac{\rho}{2}\wedge\Theta-\wedge\Theta\wedge\Theta-\frac{\rho^2}{4}I-\frac{\rho}{2}\wedge\Theta
\\&=-\frac{\rho^2}{4}I-\wedge\Theta\wedge\Theta-\rho\wedge\Theta.
\end{align*}
Similarly,
$R_1(a_1(y,x)H)|_{y=x}=-\frac{\rho^2}{4}I-\wedge\Theta\wedge\Theta-\rho\wedge\Theta$.

Summing up the above computation, we arrive at the desired value of
$a_2(x)$,
\begin{align}
a_2(x)=&\frac{1}{2}\Delta\wedge\Theta+\frac{1}{2}(\wedge\Theta)^2+\frac{\rho}{2}\wedge\Theta-\frac{1}{2}\Theta_{i\bar
j}\Theta_{j\bar i}-\frac{1}{2}R_{i\bar j}\Theta_{j\bar
i}\\&+(\frac{1}{3}\Delta\rho+\frac{1}{8}\rho^2-\frac{1}{6}|Ric|^2+\frac{1}{24}|R|^2)I,
\nonumber
\end{align}
where $\Delta\wedge\Theta=\wedge\partial\bar\partial\wedge\Theta$.

\begin{definition} {\rm (Lu \cite{Lu})}
Let $P$ be a $d$-th order covariant derivative of $R_{i\bar jk\bar
l}, R_{i\bar j},\rho,\Theta,\wedge\Theta$. Define the weight and the
order of $P$ to be the numbers $w(P)=(1+\frac d2)$ and ${\rm
ord}(P)=\frac d2$ respectively. The weight and order functions can
be extended additively to monomials of curvatures and their
derivatives.
\end{definition} For example,
$$w(R_{i\bar jk\bar
l})=1,\quad w(R_{i\bar j,k})=\frac32, \quad w(R_{i\bar j,k}R_{i\bar
j k\bar l})=\frac52,
$$
$${\rm ord}(R_{i\bar jk\bar
l})=0,\quad {\rm ord}(\rho_i R_{i\bar jk\bar l})=\frac12,\quad {\rm
ord}(\Delta\wedge\Theta)=1.$$

In the following, the word ``leading term'' will mean the sum of
terms with the highest order.
\begin{lemma} \label{coef1}
The leading term in $a_k$ is $\frac{k}{(k+1)!}\Delta^{k-1}\rho
I+\frac{1}{k!}\Delta^{k-1}\wedge\Theta$.
\end{lemma}
\begin{proof}
Let $u_k$ denote the coefficient of $\Delta^{k-1}\rho$ in
$R_{k}(1)$. Then from \eqref{eqeng},
$$R_{k}(1)=\frac{1}{\det g}\sum^{2k}_{i=k}\frac{1}{i!(i-k)!}L^{i}(\det g S^{i-k})|_{y=0},$$
here we take normal coordinate system around point $x=0$ on the
K\"ahler manifold $M$. It is not difficult to see that the only
terms that may contribute to $\Delta^{k-1}\rho$ must have $i=k$ or
$i=k+1$. So we have
$$u_k=-\frac{1}{k!}+\frac{1}{(k+1)!}=-\frac{k}{(k+1)!}.$$
From \eqref{eqb8}, we see that the coefficient of $\Delta^{k-1}\rho$
in $a_k$ equals $-u_k=\frac{k}{(k+1)!}$.

Similarly, we can prove that the coefficient of
$\Delta^{k-1}\wedge\Theta$ in $a_k$ equals $\frac{1}{k!}$.
\end{proof}

\begin{remark} Lemma \ref{coef1} was obtained by Liu and Lu \cite{LL} using peak section
method, improving the $\mathcal E=\mathbb C$ case proved in
\cite{LT}.
\end{remark}

Let $M=M_1 \times M_2$ a product of two projective manifolds
equipped with the product K\"ahler metric and the twisted bundles
$\mathcal E^{(1)}(m)\boxtimes \mathcal E^{(2)}(m)$. Denote the
Bergman kernel of $M$ by $B_m^{(M)}$. The following fact is
well-known.
\begin{lemma}\label{prodrel}
We have $B^{(M)}_m(x_1,x_2) = B^{(M_1)}_m(x_1) B^{(M_2)}_m(x_2)$.
\end{lemma}
\begin{proof}
We just need to note that if $(s_1, \cdots, s_i, \cdots)\  and \
(t_1, \cdots, t_j, \cdots)$ are orthonormal basis of $\mathcal
E^{(1)}(m)$ and $\mathcal E^{(2)}(m)$ respectively, then $(s_1 t_1,
\cdots, s_i t_j, \cdots)$ is an orthonormal basis of $\mathcal
E^{(1)}(m)\boxtimes \mathcal E^{(2)}(m)$.
\end{proof}

 Consider the semigroup $N^\infty$ of
sequences ${\bold d}=(d_1,d_2,\dots)$ where $d_i$ are nonnegative
integers and $d_i=0$ for sufficiently large $i$. We sometimes also
use $(1^{d_1}2^{d_2}\dots)$ to denote $\bold d$. We will use the
following notation:
\begin{equation} \label{eqnotation}
|\bold d|:=\sum_{i\geq 1}i\cdot d_i,\quad ||\bold
d||:=\sum_{i\geq1}d_i, \quad {\bold d}!=\prod_{i\geq1}d_i!, \quad
\bold c\cdot\bold d:=\sum_{i\geq1}c_id_i.
\end{equation}

\begin{proposition}
Let $\bold{c, d, e}\in N^\infty$ and $|\bold c|+\bold d\cdot\bold
e+||\bold e||=k$. Let $u(\bold{c,d,e})$ denote the coefficient of
the weight $k$ monomial $(\text{where }d_j\ne d_{j+1}, \forall\,
j\geq1)$
$$\prod_{i\geq1}
(\Delta^{i-1}\rho)^{c_i}\prod_{j\geq1}(\Delta^{d_j}\wedge\Theta)^{e_j}$$
in $a_k$. Then we have
\begin{equation}\label{eqb9}
u(\bold{c,d,e})=\frac{1}{\bold c!}\prod_{i\geq
1}\left(\frac{i}{(i+1)!}\right)^{c_i}\cdot \frac1{\bold
e!}\prod_{j\geq1}\left(\frac1{(d_j+1)!}\right)^{e_j}.
\end{equation}

\begin{proof}
By Lemma \ref{coef1} and Lemma \ref{prodrel}, we have the following
relations:
\begin{align*}
c_i\cdot u(\bold{c,d,e})&=\frac{i}{(i+1)!} u(\bold c-\bm{\delta}_i,\bold{d,e}), \\
e_j\cdot u(\bold{c,d,e})&= \frac{1}{(d_j+1)!}u(\bold{c,d},\bold
e-\bm{\delta}_j),
\end{align*}
where $\bm{\delta}_i\in N^\infty$ denotes the sequence with $1$ at
the $i$-th place and zeros elsewhere.

From $u(\bold{0,0,0})=1$, we can recursively prove the desired
formula \eqref{eqb9} of $u(\bold{c,d,e})$.
\end{proof}
\end{proposition}

\vskip 30pt

\section{Weyl invariants and graphs}
\label{secgraph}
In the rest of the paper, we will restrict to the case $\mathcal
E=\mathbb C$, namely $\Theta=0$.

The $a_k$'s in \eqref{eqb10} are the so-called {\it Weyl invariants}
introduced by Fefferman \cite{Fef}. Consider the tensor products of
covariant derivatives of the curvature tensor $R_{i\bar j k \bar l;
p\cdots \bar q}$, e.g.
$$R_{ijk \bar l;p \bar q}\otimes\cdots \otimes  R_{a \bar b c \bar d; \bar
e}.$$ The Weyl invariants are constructed by first pairing up the
unbarred indices to barred indices and then contracting all paired
indices.

For the sake of brevity, Weyl invariants such as
$$g^{{i_1}\bar j_1} g^{{i_2}\bar j_2}g^{{i_3}\bar j_3}g^{{i_4}\bar
j_4}g^{{i_5}\bar j_5} R_{{i_1}\bar j_{1} i_2 \bar j_3; i_5 \bar j_4}
R_{{i_3}\bar j_{2} i_4 \bar j_5 }$$ will be abbreviated as
\begin{equation} \label{eqweyl}
W:=R_{ i_1 \bar i_1 i_2 \bar i_3 ; i_5\bar i_4} R_{ i_3 \bar i_2 i_4
\bar i_5 },
\end{equation}
knowing that $(i_k, \bar i_k), 1\leq k\leq 5$ are paired indices to
be contracted.

It is useful to represent Weyl invariants as digraphs (also called
quivers), namely directed graphs possibly with loops and
multi-edges. We put curvature tensors as nodes and draw a directed
edge from $i_k$ to $\bar i_k$ for each $k$. For example, the
associated graph of the above Weyl invariant \eqref{eqweyl} is
depicted in Figure \ref{fig6}.
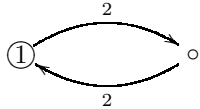
\begin{figure}[h]
$\xymatrix{
        *+[o][F-]{1} \ar@/^1pc/[rr]^2
         & &
        \circ \ar@/^1pc/[ll]^2}
$ \caption{The associated graph of $W$} \label{fig6}
\end{figure}

In this paper, we represent a multidigraph as a weighted digraph.
The weight of a directed edge is the number of multi-edges. The
number attached to a vertex denotes the number of its self-loops. A
vertex without loops will be denoted by a small circle. The indegree
and outdegree of a vertex $v$ are denoted by $\deg^-(v)$ and
$\deg^+(v)$ respectively.

By Ricci formula, we cannot recover a Weyl invariant from its
associated graph, namely different Weyl invariants may have the same
associated graph. As mentioned in Remark \ref{rm1}, a remedy for the
discrepancy is to express the Weyl invariant in terms of derivatives
of the K\"ahler metric. No information will be lost and the
uniqueness of the associated graph is guaranteed by \eqref{eqcur6}.

\begin{definition} We call a vertex $v$ of a digraph $G$ semistable if we have
$$\deg^-(v)\geq1,\ \deg^+(v)\geq1,\ \deg^-(v)+\deg^+(v)\geq3.$$
$G$ is called {\it semistable} if each vertex of $G$ is semistable.
We call $v$ stable if $\deg^-(v)\geq2,\ \deg^+(v)\geq2$. A digraph
$G$ is {\it stable} if each vertex of $G$ is stable.

Let $\mathcal G_{r,s}$ denote the set of stable digraphs with $r$
vertices and $s$ edges. The set of semistable and stable graphs will
be denoted by $\mathcal G^{ss}$ and $\mathcal G$ respectively. The
associated graph of a Weyl invariant lies in $\mathcal G$. For any
$G\in \mathcal G_{r,s}$, its weight $w(G)$ is defined to be $s-r$.
We denote by $\mathcal G^{ss}(k)$ and $\mathcal G(k)$ respectively
the set of semistable and stable digraphs with weight $k$. Let
$\mathcal G_{con}(k)$ and $\mathcal G_{scon}(k)$ respectively be the
set of connected and strongly connected graphs in $\mathcal G(k)$.
We also define a special set of graphs:
$$\Lambda(k)=\{G\in\mathcal G_{scon}(k)\mid 1 \text{ is not an
eigenvalue of } A(G)\}.$$ This set is of interest in view of
Corollary \ref{eigenvalue}. We have computed the cardinalities of
these sets when $k\leq5$ in Table \ref{tb2}.
\end{definition}
\begin{table}[h]
\caption{Numbers of stable graphs} \label{tb2}
\begin{tabular}{|c||c|c|c|c|c|}
\hline
    $k$                               &$1$&$2$&$3$ &$4$  &$5$
\\\hline           $|\mathcal G(k)|$            &$1$&$4$&$15$&$82$&$589$
\\\hline       $|\mathcal G_{con}(k)|$        &$1$&$3$&$11$&$61$&$474$
\\\hline       $|\mathcal G_{scon}(k)|$     &$1$&$3$&$10$&$51$&$373$
\\\hline     $|\Lambda(k)|$     &$1$&$3$&$9$ &$45$&$316$
\\\hline
\end{tabular}
\end{table}
It is not difficult to see that $\mathcal G(k)= \cup_{j=1}^k
\mathcal G_{j,j+k}$ for any $k\geq1$.

\begin{remark}
It shall be interesting to find formulae, closed or recursive, for
the number of graphs in the above sets.
\end{remark}

We can write the coefficient $a_k$ as a polynomial of $g_{i\bar j
\beta_1\dots\beta_{r}},\, r\geq 1$, and consequently as a sum over
graphs in $\mathcal G^{ss}(k)$,
\begin{equation} \label{eqc}
a_k(x)=\sum_{G\in\mathcal G^{ss}(k)} z(G)\cdot G, \quad
z(G)\in\mathbb Q.
\end{equation}
So we may regard $z$ as a map from $\mathcal G^{ss}$ to $\mathbb Q$.
By Remark \ref{rm1} and Lemma \ref{covar}, it is enough to know only
those $z(G)$ of stable graphs $G\in\mathcal G$.

\begin{remark}
For convenience, we assume that the empty graph $\emptyset$ is the
unique element in $\mathcal G(0)$ and $z(\emptyset)=1$. This is
consistent with $a_0=1$.
\end{remark}

First we recall Loi's recursion formula \cite{Loi}
\begin{equation} \label{eqloi}
a_k(x)=-R_k(1)-\sum_{i+j=k \atop
i,j\geq1}a_i(x)a_j(x)-\sum_{\ell+i+j=k\atop 1\leq \ell\leq
k-1}R_\ell (a_i(x,y)a_j(y,x))|_{y=x}.
\end{equation}
It was pointed out to the author recently that essentially the same identity was also obtained independently
in \cite{Cha}.

\begin{remark}
Here we outline the strategy of our proof of the closed formulae for
$z(G)$ in Theorem \ref{main}. We will give a graph-theoretic
interpretation of Loi's formula in Proposition \ref{alltyz}, the key
ingredient is the graphical formulae for derivatives of $\det g$
proved in Lemma \ref{dettree}. Then we extract the coefficient
$z(G)$ for any given graph $G$ in Lemma \ref{single}, which will be
used to prove $z(G)=0$ for any connected but not strongly connected
graph $G$ in Proposition \ref{sink}. By specializing Lemma
\ref{single} to a strongly connected graph $G$, we prove in Lemma
\ref{single2} that $z(G)$ can be written as a summation over
equivalent classes of linear subgraphs of $G$, which implies the
explicit closed formula in Proposition \ref{strong} through a
classical result in spectral graph theory: the Coefficient Theorem.
\end{remark}

Next we introduce some notation. Let $\mathscr L$ be the set of
digraphs consisting of a finite number of vertex-disjoint simple
cycles (i.e. simple polygons without common vertex). The length of a
simple cycle is defined to be the number of its edges. For each
graph $L\in \mathscr L$, we can write $L$ as a finite increasing
sequence of nonnegative integers $[i_1,\dots, i_m]$, meaning $L$
consists of $m$ disjoint simple cycles, whose lengths are specified
by $i_1,\dots,i_m$. We define the index of $L$ to be
\begin{equation}
i(L)=m+i_1+\dots+i_m.
\end{equation}
Note that $[0]$ is just a single vertex and $[1]$ is a vertex with a
self-loop. If $0\notin L$, then $L$ is usually called a linear
digraph. Recall that a {\it linear directed graph} is a digraph in
which $\deg^+(v)=\deg^-(v)=1$ for each vertex $v$.

Given a set of indices ${\alpha_1,\dots,\alpha_r}$, denote by
$\mathscr L(\alpha_1,\dots,\alpha_r)$ the isomorphism classes of all
possible decorations of the vertices of $L\in \mathscr L$ with the
half-edges ${\alpha_1,\dots,\alpha_r}$ requiring each vertex to be
semistable. Two decorations of $L$ that differ by a graph
isomorphism are considered the same.

The following crucial lemma explains the graphical properties of the
partial derivatives of $\det g$.
\begin{lemma} \label{dettree}
We have
\begin{equation} \label{eqdettree}
\frac{1}{\det g}\partial_{\alpha_1}\dots\partial_{\alpha_r}\det g
=\sum_{\substack{L\in \mathscr L(\alpha_1,\dots,\alpha_r)\\0\notin
L}} (-1)^{i(L)} \cdot L.
\end{equation}
\end{lemma}
\begin{proof}
By \eqref{eqcur7} and \eqref{eqdet}, we have $$\frac{1}{\det
g}\partial_{\alpha}\det g=g_{j\bar j\alpha}= \xymatrix{ \circ
\ar@(ul,dl)[]|{} \ar@{-}[r]^\alpha &}.$$
$$
\frac{1}{\det g}\partial_{\beta}\partial_{\alpha}\det g =g_{i\bar
i\alpha}g_{j\bar j\beta}-g_{i\bar j\alpha}g_{j\bar i\beta}+g_{j\bar
j\alpha\beta}$$\bigskip
$$=\quad \xymatrix{ \circ \ar@(ur,ul)[]
\ar@{-}[d]^\alpha & \circ \ar@(ur,ul)[] \ar@{-}[d]^\beta \\&}\quad
-\quad \xymatrix{ \circ\ar@{-}[d]_\beta \ar@/^/[r] &
\circ\ar@{-}[d]^\alpha \ar@/^/[l]\\& }\quad +  \xymatrix{ &\circ
\ar@(ur,ul)[] \ar@{-}[dr]^\alpha \ar@{-}[dl]_\beta &  \\&&}.
$$

So equation \eqref{eqdettree} follows by an easy induction.
\end{proof}

\begin{definition} \label{config}
Let $L\in\mathscr L$, $G_1, G_2 \in \mathcal G^{ss}$. We denote by
$\mathscr G(L,G_1,G_2)$ the set of all different ways to add a
finite number of directed edges to the disjoint union of three
graphs $L,G_1,G_2$ satisfying that
\begin{enumerate}
\item[i)]
No head of these new edges is connected to vertices of $G_1$ and no
tail of these new edges is connected to vertices of $G_2$. Such
$G_1$ and $G_2$ are called source and sink respectively.
\item[ii)] These new edges may act on edges of $G_1$ and $G_2$ as illustrated
in Figure \ref{fig3}.
\item[iii)] The final resulting digraph is semistable.
\end{enumerate}
\end{definition}

For $\mathcal Z\in \mathscr G(L,G_1,G_2)$, we denote by $[\mathcal
Z]$ the resulted graph. Two configurations $[\mathcal Z_1],
[\mathcal Z_2]\in \mathscr G(L,G_1,G_2)$ are considered equal if
there is an automorphism between $[\mathcal Z_1]$ and $[\mathcal
Z_2]$ leaving $L$ invariant and keeping $G_1$ and $G_2$ fixed.

Given $k\geq 1$, we also define a function
\begin{equation} \label{eqF}
F_k(L,G_1,G_2)=\sum_{\substack{\mathcal Z\in\mathscr G(L,G_1,G_2)\\
\omega([\mathcal Z])=k}}\frac{(-1)^{i(L)+1}[\mathcal Z]}{|{\rm
Aut}(\mathcal Z)|},
\end{equation}
where ${\rm Aut}(\mathcal Z)$ is the subgroup of the group of
automorphism of $[\mathcal Z]$ leaving $L$ invariant and keeping
$G_1$ and $G_2$ fixed. We may extend $F_k$ as a bilinear function in
the second and third parameters.

\begin{proposition} \label{alltyz} Given $k\geq1$, we have
\begin{equation} \label{eqgraph}
a_k(x)=\sum_{\substack{L\in\mathscr L,\ G_1,G_2\in\mathcal G^{ss} \\
\omega(G_1)<k,\ \omega(G_2)<k }} z(G_1)z(G_2) F_k(L,G_1,G_2).
\end{equation}
Note that we allow $L,G_1,G_2$ to be empty.
\end{proposition}
\begin{proof} In fact, \eqref{eqgraph} is just a graph-theoretic interpretation of
Loi's formula \eqref{eqloi}
\begin{equation}\label{eqloi2}
a_k(x)=-\sum_{\ell+i+j=k\atop i< k,\, j< k}R_\ell
(a_i(x,y)a_j(y,x))|_{y=x}.
\end{equation}
We need to take a close look at the operators $R_j$ defined in
\eqref{eqeng}.  Note that the function $S$ in \eqref{eqeng} plays
the role of an isolated vertex $[0]\in\mathscr L$, the cycle with
length zero. The linear subgraphs are contributed by derivatives of
$\det g$ as described in Lemma \ref{dettree}. Moreover, $G_1$ and
$G_2$ respectively represent $a_j(y,x)$ and $a_i(x,y)$ in the
right-hand side of \eqref{eqloi2}. The factorials in the denominator
of \eqref{eqeng} account for the removal of labels of the added
edges and isolated vertices in each configuration. With each term of
\eqref{eqeng} thus interpreted,
 \eqref{eqgraph} follows readily from \eqref{eqloi2}.
\end{proof}

If $[\mathcal Z]=G$, we call $\mathcal Z$ a $G$-configuration (abbr.
$G$-config). For a given digraph $G$, to specify a $G$-configuration
is equivalent to specify three vertex-disjoint subgraphs $L, G_1,
G_2$, where $L$ is a linear directed subgraph of $G$ without
isolated vertex and $G_1, G_2\in\mathcal G^{ss}$ are respectively a
source and a sink. Note that the union of $L, G_1, G_2$ may not
contain all vertices of $G$. Let $\mathcal L$ be the set of all
linear directed subgraphs $L$ of $G$. We may simply write $\mathcal
Z=(L,G_1,G_2)\in G\text{-config}$ and define a function
\begin{equation} \label{eqF2}
F_G(\mathcal Z)=\frac{(-1)^{i(\mathcal Z)+1}}{|{\rm Aut}(\mathcal
Z)|},
\end{equation}
where $i(\mathcal Z)$ is the index of $\mathcal Z$ defined by
\begin{equation}
i(\mathcal Z)=v(L)+p(L)+x(\mathcal Z).
\end{equation}
Here $v(L)$ is the number of vertices in $L$, $p(L)$ is the number
of components in $L$ and $x(\mathcal Z)$ is the number of vertices
of $G$ not belong to $L, G_1, G_2$.

\begin{lemma} \label{single} For any $G\in \mathcal G^{ss}$, we have
\begin{equation}
z(G)=\sum_{\mathcal Z\in G\text{-config}}z(G_1)z(G_2)F_G(\mathcal
Z),
\end{equation}
where $\mathcal Z=(L,G_1,G_2)$ runs over all isomorphism classes of
$G$-configurations.
\end{lemma}
\begin{proof}
It follows from Proposition \ref{alltyz}.
\end{proof}

Now we treat the case that $G\in \mathcal G^{ss}(k)$ is strongly
connected. In the right-hand side of \eqref{eqloi}, the graphs from
the second term are not connected. The graphs from the last term are
not strongly connected, since it contains either a source or a sink.
We see that only the first term $-R_k(1)$ contributes to $z(G)$ and
all $G$-configurations are of the form $(L,\emptyset,\emptyset)$. It
is not difficult to see that the automorphism group of the
$G$-configuration $(L,\emptyset,\emptyset)$, denoted by ${\rm
Aut_G}(L)$, is the subgroup of ${\rm Aut}(G)$ leaving $L$ invariant.
Two $G$-configurations $(L_1,\emptyset,\emptyset)$ and
$(L_2,\emptyset,\emptyset)$ are equivalent if and only if
$L_1=h(L_2)$ for some $h\in {\rm Aut}(G)$. This defines an
equivalence relation $\sim$ on $\mathcal L$. So Lemma \ref{single}
specializes to the following formula for strongly connected graphs.
\begin{lemma} \label{single2} For any strongly connected graph $G\in \mathcal G^{ss}$, we have
\begin{equation}\label{eqzstrong}
z(G)=\sum_{L\in \mathcal L/\sim}\frac{(-1)^{|V|+1+p(L)}}{|{\rm
Aut}_G(L)|},
\end{equation}
where $L$ runs over the equivalence classes of linear subgraphs of
$G$ and $p(L)$ is the number of components in $L$.
\end{lemma}

\vskip 30pt

\section{Proof of Theorem \ref{main}}

\begin{proposition} \label{disjoint}
If $G\in \mathcal G^{ss}$ is not connected and we write $G$ as a
disjoint union of connected subgraphs $G=G_1\cup\dots\cup G_m$, then
we have
\begin{equation}
z(G)=\prod_{j=1}^m z(G_j)/|Sym(G_1,\dots,G_m)|,
\end{equation}
where $Sym(G_1,\dots,G_m)$ denote the permutation group of these $m$
connected subgraphs.
\end{proposition}
\begin{proof}
It follows from Lemma \ref{prodrel}.
\end{proof}

By the above proposition, we may restrict to compute $z(G)$ for
connected $G$. For a digraph $G=(V,E)$, we can partition $V$ into
strongly connected components, namely the maximal strongly connected
subgraphs of $G$. Among these strongly connected components, we have
at least one {\it sink} (a component without outgoing edges) and one
{\it source} (a component without incoming edges).

\begin{proposition} \label{sink}
Given a stable graph $G\in \mathcal G$ which is connected but not strongly connected, then we have
$z(G) = 0$.
\end{proposition}
\begin{proof} By definition, any sink or source of a stable graph $G$ must be at
least semistable.
Without loss of generality, we may assume that $C\in\mathcal G^{ss}$ is a sink of $G$.
We
have a nonempty subgraph $G'$ of $G$ containing all vertices not in
$C$, such that
 there are $k$ arrows $e_1,\dots,e_k$ from $G'$ to $C$,
 as depicted in Figure \ref{figsink}.
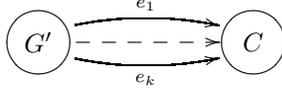
\begin{figure}[h]
$\xymatrix{ *+++[o][F-]{G'} \ar@/^/@{->} @< 3pt> [rr]^{e_1} \ar
@{-->}[rr] \ar@/_/@{->} @<-3pt> [rr]_{e_k} && *+++[o][F-]{C} }$
\caption{Decomposition of $G$ into a sink $C$ and $G'$}
\label{figsink}
\end{figure}

In this case, any $G$-configuration is a union of $\mathcal
Z'=(L',G_1',G_2')\in G'\text{-config}$ and $\mathcal
Z''=(L'',\emptyset,\emptyset)\in C\text{-config}$. Since $G$ is stable,
we can rule out Case (ii) of Definition \ref{config}.

 Let ${\rm
Aut}_f(G)$ be the subgroup of ${\rm Aut}(G)$ fixing all edges
$e_1,\dots,e_k$. Let ${\rm Aut}_{f}(C)$ and ${\rm Aut}_{f}(\mathcal
Z'')$ be respectively the subgroups of ${\rm Aut}(C)$ and ${\rm
Aut}(\mathcal Z'')$
 fixing the heads of the arrows $e_i,\,1\leq i\leq k$.
We define ${\rm Aut}_{f}(\mathcal Z')$ to be the subgroup of ${\rm
Aut}(\mathcal Z')$ fixing the tails of the arrows $e_i,\,1\leq i\leq
k$.

Given a $G'$-configuration $\mathcal Z'=(L',G_1',G_2')$ and a
$C$-configuration $\mathcal Z''=(L'',\emptyset,\emptyset)$, we
define the following set
\begin{equation}
S(\mathcal Z',\mathcal Z'')=\left\{\text{isomorphism classes in
}G\text{-config}\text{ of the form } \mathcal Z'\coprod p(\mathcal
Z''),\ p\in\frac{{\rm Aut} (C)}{{\rm Aut}_f (C)}\right\},
\end{equation}
where $p(\mathcal Z'')=(p(L''),\emptyset,\emptyset)\in
C\text{-config}$ and $\mathcal Z'\coprod p(\mathcal Z'')=(L'\coprod
p(L''),G_1',G_2')$.

First we assume that all automorphisms of $G$ fix these $e_i,\,1\leq
i\leq k$, namely ${\rm Aut}(G)= {\rm Aut}_f(G)$. By our assumption,
for $\mathcal Z=\mathcal Z'\coprod p(\mathcal Z'')\in S(\mathcal
Z',\mathcal Z'')$ we have ${\rm Aut}(\mathcal Z)={\rm
Aut}_f(\mathcal Z'){\rm Aut}_f(p(\mathcal Z''))$

It is enough to show that for any fixed $G'$-configuration $\mathcal
Z'=(L',G_1',G_2')$, the contribution to $z(G)$ from the
$G$-configurations $(L',G_1',G_2'\coprod C)$ and $\{S(\mathcal
Z',\mathcal Z'')\}_{\mathcal Z''\in C\text{-config}}$ add up to
zero.

We will proceed by induction on the weight of $G$. Namely we assume that if a semistable graph $H\in\mathcal G^{ss}$
is connected but not strongly connected and the weight of $H$
is less than the weight of $G$, then $z(H)=0$ (cf. Corollary \ref{conn}).

First we assume $G_2'=\emptyset$.
By Lemma \ref{single}, the contribution to $z(G)$ from the
$G$-configuration $\mathcal Z=(L',G_1',C)$ is
\begin{align}
&z(G_1')z(c)\frac{(-1)^{i(\mathcal Z')+1}}{|{\rm Aut}_f(\mathcal
Z')|}\times\frac{|{\rm Aut}(C)|}{|{\rm Aut}_{f}(C)|} \nonumber\\
=&z(G_1')\frac{(-1)^{i(\mathcal Z')+1}}{|{\rm Aut}_f(\mathcal
Z')|}\sum_{\mathcal Z''\in C\text{-config}}\frac{(-1)^{i(\mathcal
Z'')+1}}{|{\rm Aut}(\mathcal Z'')|}\times\frac{|{\rm Aut}(C)|}{|{\rm
Aut}_f(C)|}. \label{cont1}
\end{align}

The contribution to $z(G)$ from the set of $G$-configurations
$\{S(\mathcal Z',\mathcal Z'')\}_{\mathcal Z''\in C\text{-config}}$
is
\begin{equation} \label{cont2}
z(G_1')\frac{1}{|{\rm Aut}_f(\mathcal Z')|}\sum_{\substack{\mathcal
Z''\in C\text{-config}\\\mathcal Z'\coprod p(\mathcal Z'')\in
S(\mathcal Z',\mathcal Z'')}} \frac{(-1)^{i(\mathcal Z')+i(\mathcal
Z'')+1}}{|{\rm Aut}_f(p(\mathcal Z''))|}.
\end{equation}

We now show that the contributions to $z(G)$ in \eqref{cont1} and
\eqref{cont2} add up to zero, namely we need to prove that for any
$\mathcal Z'' \in C\text{-config}$
\begin{equation} \label{eqkey}
\frac{1}{|{\rm Aut}(\mathcal Z'')|}\times\frac{|{\rm Aut}(C)|}{|{\rm
Aut}_f(C)|}=\sum_{\mathcal Z'\coprod p(\mathcal Z'')\in S(\mathcal
Z',\mathcal Z'')}\frac{1}{|{\rm Aut}_f(p(\mathcal Z''))|}.
\end{equation}
This can be seen as follows: Denote by $H$ the set of all possible
distributions of the labeled heads of $e_i,\,1\leq i\leq k$ on $C$.
Note that some of these $e_i$ may have the same head on $C$. It is
obvious that $H$ is in one-to-one correspondence with ${\rm Aut}(C)/
{\rm Aut}_{f}(C)$.

We have the natural action of ${\rm Aut}(\mathcal Z'')$ on $H$, then
the set of orbits is just $S(\mathcal Z',\mathcal Z'')$ and the
isotropy group at $\mathcal Z=\mathcal Z'\coprod p(\mathcal Z'')\in
S(\mathcal Z',\mathcal Z'')$ is just ${\rm Aut}_f(p(\mathcal Z''))$.

If $G_2'\neq\emptyset$ and $G_2'$ is connected with a tail of some
$e_i$ ,then the contribution to $z(G)$ from the $G$-configurations
$(L',G_1',G_2'+ C)$ is zero by induction and all configurations in
$\{S(\mathcal Z',\mathcal Z'')\}_{\mathcal Z''\in C\text{-config}}$
are not allowed. So we may assume that there are no edges between
$G_2'$ and $C$, then we have the same cancellation by Proposition
\ref{disjoint}. Note that $G_2'+ C$ denote the induced subgraph whose
vertices are the union of vertices of $G_2'$ and $C$.

If ${\rm Aut}(G)\neq {\rm Aut}_f(G)$, the above argument still works
with appropriate modifications.
\end{proof}

\begin{remark}
We will give an alternative proof of Proposition \ref{sink} in
\cite{Xu}. As pointed out by the referee, Proposition \ref{sink}
should also follow from the path integral formula by Douglas and
Klevtsov \cite{DK}.
\end{remark}

\begin{corollary} \label{conn}
If $G\in \mathcal G^{ss}$ is connected but not strongly connected, then
$z(G) = 0$.
\end{corollary}
\begin{proof}
Given any stable graph $H\in \mathcal G$, if we apply Lemma \ref{covar} to each vertices of $H$, we get a sum of
semistable graphs of the same weight. Moreover, it is not difficult to see that if $H$ is strongly connected, all these
semistable graphs are also strongly connected. So we conclude the proof.
\end{proof}

\begin{example}
Let us illustrate the above proof by considering the following
semistalbe graph $G$ in Figure \ref{figzero}, which is only weakly
connected.
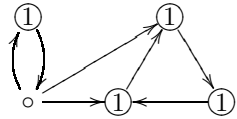
\begin{figure}[h]
$\xymatrix@C=3mm@R=7mm{
            *+[o][F-]{1}   \ar@/^0.5pc/[d]
                 & & &*+[o][F-]{1}  \ar[dr]&  \\
             \circ\ar[urrr]  \ar@/^0.5pc/[u] \ar[rr]
                & & *+[o][F-]{1}\ar[ur] && *+[o][F-]{1}   \ar[ll]     } $
                \caption{A semistable graph $G\in \mathcal G^{ss}(6)$}
                \label{figzero}
\end{figure}
Let $C$ be its unique sink and $G'$ be its unique source. Then we
have $|{\rm Aut}(C)|=3$ and $|{\rm Aut}_f(C)|=1$. For all
$G$-configurations, we can check that the two sides of the equation
\eqref{eqkey} match, their values are listed in Table \ref{tb1}.
\begin{table}[h]
\caption{$G$-configurations}\label{tb1}
\begin{tabular}{|c|c|c|c|c|c|}
\hline \backslashbox{$L'$}{$L''$}
 & $\emptyset$ & $[1]$ & $[1^2]$ & $[1^3]$ & $[3]$
\\
\hline $\emptyset$ & 1 & 3 & 3 & 1 & 1\\
\hline $[1]$       & 1 & 3 & 3 & 1 & 1\\
\hline $[2]$       & 1 & 3 & 3 & 1 & 1\\
\hline
\end{tabular}
\end{table}
\end{example}

The following theorem is called the coefficient theorem (see Theorem
1.2 in \cite{CDS}) from the spectral graph theory.
\begin{theorem} \label{cds}
Let $P_G (\lambda) = \lambda^n + c_1\lambda^{n-1} + \dots + c_n$
 be
the characteristic polynomial of a digraph $G$ with $n$ vertices.
Then for each $i = 1,\dots,n$,
$$c_i = \sum_{L\in\mathcal L_i} (-1)^{p(L)}, $$ where $\mathcal L_i$ is the set of all
linear directed subgraphs $L$ of $G$ with exactly $i$ vertices;
$p(L)$ denotes the number of components of $L$.
\end{theorem}

\begin{proposition} \label{strong}
If $G=(V,E)\in \mathcal G^{ss}$ is strongly connected, then
\begin{equation} \label{eqstrong}
 z(G) = (-1)^{|V|+1}\frac{\det(I-A)}{|{\rm Aut}(G)|}.
\end{equation}

\end{proposition}
\begin{proof}
Let $\mathcal L$ be the set of all linear directed subgraphs of $G$.
We have the natural action of ${\rm Aut}(G)$ on $\mathcal L$, given
by $h\cdot L=h(L)$ for $h\in {\rm Aut}(G), L\in\mathcal L$. From our
discussion at the end of Section \ref{secgraph}, the orbits of this
action are just $\mathcal L/\sim$ and the isotropy group at
$L\in\mathcal L$ is just ${\rm Aut}_G(L)$. By Theorem \ref{cds} and
Lemma \ref{single2}, we have
\begin{align*}
\det(I-A)=1+\sum_{i=1}^n c_i &=\sum_{L\in\mathcal
L/\sim}(-1)^{p(L)}|{\rm Aut}(G)\cdot
L|\\
&=|{\rm Aut}(G)|\sum_{L\in\mathcal L/\sim}\frac{(-1)^{p(L)}}{|{\rm
Aut}_G(L)|}\\
&=|{\rm Aut}(G)| (-1)^{|V|+1}z(G),
\end{align*}
which is just the equation \eqref{eqstrong}.
\end{proof}

\begin{corollary} \label{eigenvalue}
If $G \in \mathcal G^{ss}$ is strongly connected, then $z(G) = 0$ if
and only if $1$ is an eigenvalue of the characteristic polynomial of
$G$.
\end{corollary}
\begin{proof}
It follows from Proposition \ref{strong}.
\end{proof}

Propositions \ref{disjoint}, \ref{sink} and
\ref{strong} together imply Theorem \ref{main}.

\begin{proposition} \label{onevertex} If $G=\xymatrix{*+[o][F-]{k}}$, the digraph with
one vertex and $k\geq2$ self-loops, we have
\begin{equation}
z(G)=-\frac{k-1}{k!}.
\end{equation}
\end{proposition}
\begin{proof}
It follows from Lemma \ref{coef1} or Proposition \ref{strong}.
\end{proof}

\begin{proposition} \label{twovertex}
Let  $G=\xymatrix{
        *+[o][F-]{m} \ar@/^/[r]^i
         &
        *+[o][F-]{n} \ar@/^/[l]^j}
$ be a strongly connected stable graph with two vertices, namely
$ij\neq0$. Then we have
\begin{equation}
z(G) = \begin{cases} \dfrac{ij-(1-m)(1-n)}{2\cdot m!\, n!\, i!\, j!}
&\mbox{if } i=j \text{ and } m=n;
\\ \dfrac{ij-(1-m)(1-n)}{ m!\, n!\, i!\, j!} & \mbox{otherwise} .
\end{cases}
\end{equation}
\end{proposition}
\begin{proof}
The formula follows from Proposition \ref{strong}.
\end{proof}

\begin{example} \label{tyza2}
Fix a normal coordinate around $x\in M$. By Proposition
\ref{onevertex}, we get the well known
\begin{equation}\label{eqa1}
a_1=-\frac{1}{2}\left[\xymatrix{*+[o][F-]{2}}\right]=-\frac{1}{2}g_{i\bar
i j\bar j}=\frac{1}{2}R_{i\bar i j\bar j}=\frac{1}{2}\rho,
\end{equation}
where $(i,\bar i),\,(j, \bar j)$ are paired indices to be
contracted.

By Proposition \ref{onevertex}, Proposition \ref{twovertex} and
Proposition \ref{disjoint}, we have
\begin{equation}\label{eqa2}
a_2=-\frac{1}{3}\left[\xymatrix{*+[o][F-]{3}}\right]+\frac{1}{2}\left[\xymatrix{
        *+[o][F-]{1} \ar@/^/[r]^1
         &
        *+[o][F-]{1} \ar@/^/[l]^1} \right ]+ \frac{3}{8}
         \left[\xymatrix{ \circ  \ar@/^/[r]^2 & \circ \ar@/^/[l]^2
         }\right]+\frac{1}{8}\left[\xymatrix{*+[o][F-]{2}}\mid \xymatrix{*+[o][F-]{2}}\right].
\end{equation}
Written in terms of derivatives of the K\"ahler metric $g_{i\bar
j}$, we get
\begin{equation} \label{eqb11}
a_2=-\frac{1}{3}g_{i\bar i j\bar j k\bar k} +\frac{1}{2} g_{i\bar i
k\bar l}g_{j\bar j l\bar k}+\frac{3}{8}g_{i\bar j k\bar l}g_{j\bar i
l\bar k}+\frac{1}{8}g_{i\bar i j\bar j}g_{k\bar k l\bar l}.
\end{equation}
It is understood that $(i,\bar i),\,(j, \bar j),\,(k,\bar
k),\,(l,\bar l)$ are paired indices to be contracted.

Apply the operator $D$ (defined in Remark \ref{rm1}) to the
right-hand side of \eqref{eqb11} and use the identities
$$D(g_{i\bar j k\bar l})=-R_{i\bar j k\bar l},$$
\begin{align*}
D(g_{i\bar i j\bar j k\bar k})&=-R_{i\bar i j \bar j; k\bar
k}+R_{k\bar i s\bar j}R_{i\bar k j\bar s}+R_{j\bar j s\bar
i}R_{k\bar k i\bar s}+R_{i\bar i s\bar j}R_{k\bar k j\bar s}\\
&=-\Delta \rho+|R|^2+2|Ric|^2,
\end{align*}
we arrived at
\begin{equation}
a_2=\frac{1}{3}\Delta
\rho+\frac{1}{24}|R|^2-\frac{1}{6}|Ric|^2+\frac{1}{8}\rho^2,
\end{equation}
which is the same as the function $a_2(x)$ computed in \cite{Lu}. We
computed $a_3$ in Appendix \ref{apthree} and we can use tables at
Appendix \ref{apfour} to compute $a_4$.
\end{example}

\vskip 30pt

\section{Computations of $z(G)$}
In this section, we derive some explicit formulae of $z(G)$ .

\begin{figure}[h]
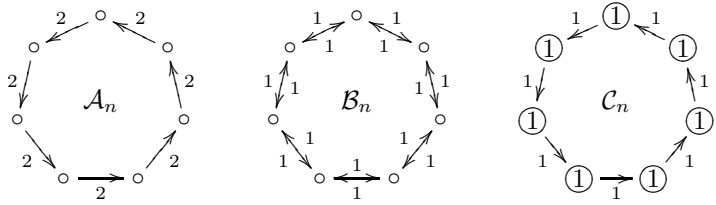
$$
\xy/r2.7pc/: {*{\mathcal A_n}{}\xypolygon7{~><{} ~*{\circ}
~>>{_{2}}}}
\endxy \qquad \xy/r2.7pc/: {*{\mathcal B_n}{}\xypolygon7{~><{@{<->}} ~*{\circ} ~>>{_{1}^{1}}}}
\endxy \qquad \xy/r2.7pc/: {*{\mathcal C_n}{}\xypolygon7{~><{} ~*{\xybox{*{1}*\cir<2mm>{}}}
~>>{_{1}}}}
\endxy$$
\caption{Three types of stable graphs in $\mathcal G_{n,2n}$}
\label{figthree}
\end{figure}

\begin{proposition} Given $n\geq3$, for the three graphs in Figure \ref{figthree}, we have
\begin{align*}
z(\mathcal A_n)&=\frac{(-1)^n(2^n-1)}{2^n
n},\\
z(\mathcal
B_n)&=\frac{(-1)^n}{2n}+\sum_{i=0}^{[n/2]}(-1)^{n+i+1}\sum _{H\in
\mathscr C_{i,n-2i}} \frac{1}{{\rm Aut}(H)}\\&=\begin{cases}0, &
n\equiv 0\\(-1)^n/(2n), & n\equiv \pm 1\\3(-1)^n/(2n), & n\equiv \pm
2
\\2(-1)^n/n, &n\equiv 3 \end{cases}\mod\, 6,\\
z(\mathcal C_n)&=\frac{(-1)^n}{n}.
\end{align*}
Here $\mathscr C_{i,j}$ denote the set of $2$-colored undirected
cycle graphs with $i$ black vertices and $j$ white vertices and
${\rm Aut}(H)$ is the group of color-preserving automorphisms. Note the first equation for $z(\mathcal
B_n)$ holds only when $n\geq5$.
\end{proposition}
\begin{proof} Now we will compute $z(G)$ for $G=\mathcal A_n, \mathcal B_n, \mathcal C_n$.
Since these graphs are strongly connected, all $G$-configurations
are simply linear subgraphs $L$ of $G$. Two linear subgraphs $L_1$
and $L_2$ define the same $G$-configuration if there is an
automorphism $\varphi$ of $G$ such that $\varphi(L_1)=L_2$.

i) When  $G=\mathcal A_n$, it is easy to see that there are only two
$\mathcal A_n$-configurations given by empty graph
$\emptyset$ and a $n$-cycle $[n]$, respectively. By \eqref{eqzstrong}, we have
\begin{align*}
z(\mathcal A_n)=\frac{(-1)^{n+1}}{|{\rm Aut}(\mathcal A_n)|}+\frac{(-1)^n}{|{\rm Aut}_{\mathcal A_n}([n])|}
&=\frac{(-1)^{n+1}}{2^n n}+\frac{(-1)^n}{n}\\
&=\frac{(-1)^n(2^n-1)}{2^n n}.
\end{align*}

ii) When $G=\mathcal B_n$, the contribution to $z(\mathcal B_n)$ by empty graph is
\begin{equation*}
\frac{(-1)^{n+1}}{{\rm Aut}(\mathcal B_n)}=\frac{(-1)^{n+1}}{2n},
\end{equation*}
where ${\rm Aut}(\mathcal B_n)$ is the dihedral group.

The contribution to $z(\mathcal B_n)$ by a $n$-cycle $[n]$ is
\begin{equation*}
\frac{(-1)^n}{|{\rm Aut}_{\mathcal B_n}([n])|}=\frac{(-1)^n}{n}.
\end{equation*}

All other $\mathcal B_n$-configurations are given by disjoint cycles of
length two
\begin{equation*}
[2^i], \quad 0\leq i\leq [n/2],
\end{equation*}
which are in one-to-one
correspondence with the set $\mathscr C_{i,n-2i}$ consisting of
$2$-colored undirected cycle graphs with $i$ black vertices and
$n-2i$ white vertices. Their contributions to $z(\mathcal B_n)$ are
equal to
\begin{equation*}
(-1)^{n+i+1}\sum _{H\in \mathscr C_{i,n-2i}} \frac{1}{{\rm Aut}(H)},
\end{equation*}
where ${\rm Aut}(H)$ is the group of color-preserving automorphism. Summing up, we get the first
equation for $z(\mathcal B_n)$. The second equation follows from a
computation of $\det(I-A(\mathcal B_n))$.

iii) When $G=\mathcal C_n$, we use $\det(I-A(\mathcal
C_n))=-1$.
\end{proof}

\begin{proposition} Let $n,m\geq2$.
\begin{enumerate}
\item[i)]
Let $K_n$ be the complete digraph with $n$ vertices and every pair
of vertices is an edge, including a loop at each vertex. Namely
every entry of the adjacency matrix of $K_n$ is $1$. Then
\begin{equation*}
z(K_n)=(-1)^{n}(n-1)/n!.
\end{equation*}
\item[ii)]
Let $D_n$ be the de Bruijn graph of degree $n$. The graph $D_n$ has
$2^{n-1}$ vertices, which are the sequences of $0$'s and $1$'s with
length $n-1$. There is an edge from $a_1a_2\dots a_{n-1}$ to
$b_1b_2\dots b_{n-1}$ if and only if $a_2a_3\dots
a_{n-1}=b_1b_2\dots b_{n-2}$. Then
\begin{equation*}
 z(D_n)=1/2.
\end{equation*}
\item[iii)] Let $K_{m,n}(V_1,V_2)$ be the complete bipartite digraph with $|V_1|=m,|V_2|=n$. Every vertex
in $V_1$ has an arrow to every vertex in $V_2$ and vice versa. Then
$$z(K_{m,n})=\frac{mn-1}{(1+\delta_{m,n})m!n!},$$
where $\delta_{m,n}=0$ if $m\neq n$ and $\delta_{n,n}=1$.
\end{enumerate}
\end{proposition}
\begin{proof}
First note that $K_n\in\mathcal G(n^2-n)$, $D_n\in\mathcal
G(2^{n-1})$ and $K_{m,n}\in\mathcal G(2mn-m-n)$ all are strongly connected digraphs. It is well-known that
their characteristic polynomials are respectively
given by
\begin{gather}
\det(\lambda I-A(K_n))=\lambda^{n-1}(\lambda-n),\\
\det(\lambda I-A(D_n))=\lambda^{2^{n-1}-1}(\lambda-2),\\
\det(\lambda I-A(K_{m,n}))=(-1)^{m+n}\lambda^{m+n-2}(\lambda^2-mn).
\end{gather}
We also have $|{\rm Aut}(K_n)|=n!$, $|{\rm Aut}(D_n)|=2$ and $|{\rm Aut}(K_{m,n})|=(1+\delta_{m,n})m!n!$, we conclude
the proof by Proposition \ref{strong}.
\end{proof}

\begin{remark}
The computation of $z(G)$ is reduced to study the spectra and the
automorphism group of $G$. A detailed study of the spectral
properties of graphs can be found at \cite{CDS}, which also contains
a chapter on the relations between spectra of a graph and its
automorphism group.
\end{remark}

Let $U$ be the unit ball of $\mathbb C^N$ and $g_{i\bar j}$ is the
Bergman metric on it,
\begin{equation}
g_{i\bar j}(z)=\frac{(1-|z|^2)\delta_{ij}+\bar z_i
z_j}{(1-|z|^2)^2},
\end{equation}
we know that $U$ is a normal coordinate with center $0$ for
$g_{i\bar j}$ and the curvature is constant
\begin{equation} \label{equ1}
R_{i\bar j k\bar l}=-(g_{i\bar j}g_{k\bar l}+g_{i\bar l}g_{k\bar
j}).
\end{equation}

The weighted reproducing kernel for $(U,g_{i\bar j})$ is
\begin{equation}
K_\beta(z)=\frac{\Gamma(\beta)}{\pi^N\Gamma(\beta-N)}(1-|z|^2)^{-\beta}.
\end{equation}
We know that $K_\beta(z)$ has an asymptotic expansion \cite{Ber,
Eng}
\begin{equation}
K_\beta(z)=(\beta/\pi)^N e^{\beta\Phi(z)} \sum_{k=0}^\infty a_k(z)
\beta^{-k},
\end{equation}
where $a_k(z)$ is the $k$-th coefficient of Tian-Yau-Zelditch
expansion and the potential $\Phi(z)$ is given by $
\Phi(z)=-\log(1-|z|^2)$.

So in the case of unit ball equipped with the Bergman metric, $a_k$
equals the polynomial
\begin{equation}
P_k=\sum _{1\leq i_1< i_2<\dots < i_k \leq N} (-i_1)\cdots(-i_k).
\end{equation}
In particular,
$$P_1=-\frac{N^2}{2}-\frac{N}{2},\qquad P_2=\frac{N^4}{8}+\frac{N^3}{12}-\frac{N^2}{8}-\frac{N}{12}.$$

\begin{lemma} \label{bergball}
Let $g_{i\bar j}$ be the Bergman metric on the unit ball $U$ of
$\mathbb C^N$. Then $g_{i\bar j \alpha_1\alpha_2\dots\alpha_r}(0)$
is nonzero only if the number of barred and unbarred indices in
$\{\alpha_1,\alpha_2,\dots,\alpha_r\}$ are equal. In this case, we
have
\begin{equation}\label{equ2}
g_{i_1\bar j_1 i_2 \bar j_2\dots i_k \bar
j_k}(0)=(k-1)!\sum_{\sigma\in S_k} g_{i_1 \bar j_{\sigma(1)}}g_{i_2
\bar j_{\sigma(2)}}\dots g_{i_k \bar j_{\sigma(k)}}(0).
\end{equation}
\end{lemma}

\begin{proof} The first assertion can be proved by an induction. Let
us prove \eqref{equ2}. By \eqref{eqcur1} and \eqref{equ1}, on $U$ we
have
\begin{equation}\label{equ3}
g_{i_1\bar j_1 i_2\bar j_2}=g_{i_1\bar j_1}g_{i_2\bar
j_2}+g_{i_1\bar j_2}g_{i_2\bar j_1}+g^{m\bar n}g_{m\bar j_1\bar
j_2}g_{i_1\bar n i_2}.
\end{equation}
As we have discussed in Section \ref{sectensor}, both sides of the
above equation may be represented by trees with half-edges and
taking partial derivatives may be regarded as the action of
half-edges on trees.

Let us look at the coefficients of $g_{i_1 \bar j_{1}}g_{i_2 \bar
j_{2}}\dots g_{i_k \bar j_{k}}$ after taking partial derivatives
$\partial_{i_3\bar j_3\dots i_k\bar j_k}$ to the right-hand side of
\eqref{equ3}. In fact, all contributions come from taking paired
derivatives $\partial_{i_3\bar j_3}\dots \partial_{i_k\bar j_k}$
consecutively to the first term in the right-hand side of
\eqref{equ3}. The coefficient is easily seen to be $2\cdot 3 \cdots
(k-1)=(k-1)!$.
\end{proof}

A {\it cycle decomposition} of a digraph $G$ is a partition of the
edges of $G$ into edge-disjoint cycles (i.e. closed directed paths
having no common edge). Let $\mathscr C_G$ denote the set of all
cycle decomposition of $G$. It is well known that $G$ admits a cycle
decomposition if and only if it $\deg^+(v)=\deg^-(v)$ for each
vertex $v$ of $G$.

\begin{proposition}
Given $k\geq1$, we have
\begin{equation}
\sum_{G\in\mathcal G(k)} (-)^{n(G)}\frac{\det(A-I)}{|{\rm Aut}(G)|}
\prod_{v\in V}(\deg^+(v)-1)!\sum_{H\in \mathscr C_G} N^{p(H)} =P_k,
\end{equation}
where $v$ runs over all vertices of $G$, $n(G)$ is the number of
components of $G$ and $p(H)$ is the number of cycles in the
cycle decomposition $H\in \mathscr C_G$. The graph $G$ in the
left-hand side need only run over all balanced graphs.
\end{proposition}
\begin{proof} From $a_k=P_k$ and Lemma \ref{bergball}, we have
\begin{equation}
\sum_{G\in\mathcal G(k)} z(G)\prod_{v\in V}(\deg^+(v)-1)!\sum_{H\in
\mathscr C_G} N^{p(H)} =P_k.
\end{equation}
Note that if $G\in\mathcal G$ is a disjoint union of connected
subgraphs $G=G_1\cup\dots\cup G_n$, then by Theorem \ref{main} we
have
\begin{equation}
z(G)=\begin{cases}\frac{(-1)^n\det(A-I)}{|{\rm Aut}(G)|}, & \text{if all }
G_i \text{ are strongly connected};
\\
0,& \text{otherwise}.
\end{cases}
\end{equation}
So we get the desired formula.
\end{proof}

\begin{corollary}
The leading term of the polynomial $P_k$ is $ \frac{(-1)^k}{2^k
k!}N^{2k} $.
\end{corollary}
\begin{proof}
It follows from Proposition \ref{disjoint} and Proposition
\ref{onevertex}.
\end{proof}

Let $G$ be a digraph. An {\it Euler tour} is a directed closed path
in $G$ which visits each edge exactly once. As a generalization of
the ``Seven Bridges of K\"onigsberg'' problem, Euler showed that $G$
has an Euler tour if and only if $G$ is connected and
$\deg^+(v)=\deg^-(v)$ at every vertex $v$. We denote by
$\epsilon(G)$ the number of Euler tours in $G$ starting with a fixed edge.

\begin{corollary} \label{bern}
Given $k\geq1$, we have
\begin{equation} \label{eqeuler}
\sum_{G\in\mathcal G(k)} z(G)\cdot\epsilon(G)\cdot\prod_{v\in
V}(\deg^+(v)-1)!
 =\frac{(-1)^{k+1}}{k}B_k,
\end{equation}
where $B_k$ is the $k$-th Bernoulli number: $B_0=1, B_1=-1/2,
B_2=1/6, B_3=0$.
\end{corollary}
\begin{proof} Using Barnes' asymptotic formula for $\Gamma$
functions, we have the following equality of power series (cf.
\cite{Eng})
\begin{equation}
\sum_{k=0}^\infty P_k x^k=\exp(q_1 x+q_2 x^2+\dots),
\end{equation}
where $q_1, q_2,\dots$ are polynomials in $N$,
\begin{equation}
q_j=\frac{1}{j(j+1)}\sum_{i=0}^j (-1)^{i+1}\binom{j+1}{i} B_i
N^{j+1-i}.
\end{equation}

Note the left-hand side of \eqref{eqeuler} is the coefficient of $N$
in $a_k$, which is equal to the coefficient of $N$ in $q_k$. The
latter is just $\frac{(-1)^{k+1}}{k}B_k$.
\end{proof}
It is interesting to note that the factor $\prod_{v\in
V}(\deg^+(v)-1)!$ also appears in the formula (cf. \cite[p.56]{Sta})
$$\epsilon(G)=\tau(G)\prod_{v\in
V}(\deg^+(v)-1)!,$$ where $\tau(G)$ is the number of oriented spanning
subtrees of $G$ with a fixed root. An oriented tree $T$ with a root $v$ means that
the underlying undirected graph is a tree and all arrows in $T$ point toward $v$.

Let us check \eqref{eqeuler}.
When $k=1$ and $k=2$, by \eqref{eqa1} and \eqref{eqa2}, the left-hand side of \eqref{eqeuler} respectively equals to
\begin{gather*}
-\frac{1}{2}\cdot1\cdot1=-\frac{1}{2}=B_1,\\
-\frac{1}{3}\cdot2\cdot2+\frac{1}{2}\cdot1\cdot1+\frac{3}{8}\cdot2\cdot1=-\frac{1}{12}=\frac{-B_2}{2}.
\end{gather*}
When $k=3$, by \eqref{eqa3} in Appendix \ref{apthree}, we have
\begin{gather*}
z_4\cdot1\cdot1+z_5\cdot2\cdot1+z_6\cdot1\cdot1+z_7\cdot3\cdot1+z_9\cdot2\cdot2\\
+z_{10}\cdot4\cdot2+z_{14}\cdot6\cdot6+z_{15}\cdot4\cdot1
=0=\frac{B_3}{3}.
\end{gather*}

\vskip 30pt

\appendix

\section{Computations of $a_3$} \label{apthree}

We will express $a_3$ in terms of the following basis as used by
Engli\v{s} \cite{Eng}.
\begin{gather*}
\sigma_1=\rho^3, \quad \sigma_2=\rho R_{i\bar j}R_{j\bar i}, \quad
\sigma_3=\rho R_{i\bar j k\bar l}R_{j\bar i l\bar k}, \\
\sigma_4=R_{i\bar j}R_{k\bar l} R_{j\bar i l\bar k},\quad
\sigma_5=R_{i\bar j}R_{k\bar i l\bar m}R_{j\bar k m\bar l},\quad
\sigma_6=R_{i\bar j}R_{j\bar k}R_{k\bar i},\\
\sigma_7=R_{i\bar j k\bar l}R_{j\bar i m\bar n}R_{l\bar k n\bar
m},\quad \sigma_8=\rho\Delta\rho,\quad \sigma_9=R_{i\bar j}R_{j\bar
i;k\bar k},\\
\sigma_{10}=R_{i\bar j k\bar l}R_{j\bar i l\bar k;m\bar m},\quad
\sigma_{11}=\rho_{;i}\rho_{;\bar i},\quad \sigma_{12}=R_{i\bar j;
k}R_{j\bar i;\bar k},\\
\sigma_{13}=R_{i\bar j k\bar l;m}R_{j\bar i l\bar k;\bar m},\quad
\sigma_{14}=\Delta^2 \rho,\quad \sigma_{15}=R_{i\bar j k\bar
l}R_{j\bar m l \bar n}R_{m \bar i n\bar k}.
\end{gather*}

We will compute the coefficients $c_i,\,1\leq i \leq 15$, such that
\begin{equation} \label{eqthreeinr}
a_3=c_1\sigma_1+c_2\sigma_2+\cdots+c_{15}\sigma_{15}.
\end{equation}

There are $15$ stable graphs of weight $3$ in $\mathcal G(3)$.
\begin{gather*}
\tau_1=\left[\xymatrix{*+[o][F-]{2}}\mid \xymatrix{*+[o][F-]{2}}\mid
\xymatrix{*+[o][F-]{2}}\right],\quad
\tau_2=\left[\xymatrix@C=6mm{
        *+[o][F-]{1} \ar@/^/[r]^1
         &
        *+[o][F-]{1} \ar@/^/[l]^1} \,\Big{|}\, \xymatrix{*+[o][F-]{2}}\right
        ],\quad
\tau_3=\left[\xymatrix@C=6mm{ \circ  \ar@/^/[r]^2 & \circ
\ar@/^/[l]^2} \,\Big{|}\, \xymatrix{*+[o][F-]{2}}
         \right],\\
\tau_4=\left[\begin{minipage}{0.63in}
             $\xymatrix@C=2mm@R=5mm{
                & \circ \ar@/^-0.3pc/@<0.2ex>[dl]^{1} \ar@/^0.3pc/@<0.7ex>[dr]^{1}             \\
          *+[o][F-]{1} \ar@/^0.3pc/@<0.7ex>[ur]^{1}  & &
          *+[o][F-]{1} \ar@/^-0.3pc/@<0.2ex>[ul]^{1}
         }$
         \end{minipage}\right],\quad
\tau_5=\left[\begin{minipage}{0.6in}$\xymatrix@C=2mm@R=5mm{
                & *+[o][F-]{1} \ar[dr]^{1}             \\
         \circ \ar[ur]^{1} \ar@/^0.3pc/[rr]^{1} & &  \circ  \ar@/^0.3pc/[ll]^{2}
         }$
         \end{minipage}\right],\quad
\tau_6=\left[\begin{minipage}{0.6in}$\xymatrix@C=2mm@R=5mm{
                & *+[o][F-]{1} \ar[dr]^{1}             \\
         *+[o][F-]{1} \ar[ur]^{1}  & &  *+[o][F-]{1} \ar[ll]^{1}
         }$
         \end{minipage}\right],\\
\tau_7=\left[\begin{minipage}{0.63in}
             $\xymatrix@C=2mm@R=5mm{
                & \circ \ar@/^-0.3pc/@<0.2ex>[dl]^{1} \ar@/^0.3pc/@<0.7ex>[dr]^{1}             \\
         \circ \ar@/^0.3pc/@<0.7ex>[ur]^{1} \ar@/^-0.3pc/@<0.2ex>[rr]^{1} & &  \circ
         \ar@/^0.3pc/@<0.7ex>[ll]^{1} \ar@/^-0.3pc/@<0.2ex>[ul]^{1}
         }$
         \end{minipage}\right],\quad
\tau_8=\left[\xymatrix{*+[o][F-]{2}}\mid
\xymatrix{*+[o][F-]{3}}\right],\quad
\tau_9=\left[\xymatrix{
        *+[o][F-]{1} \ar@/^/[r]^1
              &
        *+[o][F-]{2} \ar@/^/[l]^1}\right],\\
\tau_{10}=\left[\xymatrix{
        *+[o][F-]{1} \ar@/^/[r]^2
              &
        \circ \ar@/^/[l]^2}\right],\quad
\tau_{11}=\Big[\xymatrix{*+[o][F-]{2} \ar[r]^{1} &
*+[o][F-]{2}}\Big],\quad
\tau_{12}=\left[\xymatrix{
        *+[o][F-]{1} \ar@/^/[r]^1
              &
        *+[o][F-]{1} \ar@/^/[l]^2}\right],\\
\tau_{13}=\left[\xymatrix{ \circ  \ar@/^/[r]^3 & \circ
\ar@/^/[l]^2}\right],\quad
\tau_{14}=\left[\xymatrix{*+[o][F-]{4}}\right],\quad
\tau_{15}=\left[\begin{minipage}{0.6in}$\xymatrix@C=2mm@R=5mm{
                & \circ \ar[dr]^{2}             \\
         \circ \ar[ur]^{2}  & &  \circ \ar[ll]^{2}
         }$
         \end{minipage}\right].
\end{gather*}

By Proposition \ref{onevertex}, Proposition \ref{twovertex} and
Proposition \ref{disjoint}, we have
\begin{equation} \label{eqthreeing}
a_3=\sum_{i=1}^{15}z(\tau_i)\tau_i=
z_1\tau_1+z_2\tau_2+\cdots+z_{15}\tau_{15},
\end{equation}
where
\begin{gather}
z_1=-1/48,\quad z_2=-1/4,\quad
z_3=-3/16,\quad z_4=0,\quad z_5=-1,\nonumber\\
z_6=-1/3,\quad z_7=-2/3,\quad z_8=1/6,\quad
z_9=1/2,\quad z_{10}=1,\label{eqa3}\\
z_{11}=0,\quad z_{12}=1,\quad z_{13}=5/12,\quad z_{14}=-1/8,\quad
z_{15}=-7/24.\nonumber
\end{gather}

We need to express each $\tau_i$ as a linear combination of
$\sigma_i,\,1\leq i\leq 15$. By a standard computation as in Example
\ref{tyza2}, we get
\begin{gather}
\tau_i = -\sigma_i,\quad 1\leq i\leq 7,\nonumber\\
\tau_8 = -2\sigma_2-\sigma_3+\sigma_8,\quad
\tau_9 =
-\sigma_4-\sigma_5-\sigma_6+\sigma_9,\quad
\tau_{10} =
-2\sigma_5+\sigma_{10}-\sigma_{15}, \nonumber\\
\tau_{11} = \sigma_{11},\quad \tau_{12} = \sigma_{12},\quad
\tau_{13} =\sigma_{13},\label{eqsz}\\
\quad \tau_{14} =
-3\sigma_4-12\sigma_5-3\sigma_6+6\sigma_7+7\sigma_9+8\sigma_{10}
+10\sigma_{12}+3\sigma_{13}-\sigma_{14}-6\sigma_{15},\nonumber\\\nonumber
\tau_{15} = -\sigma_{15}.
\end{gather}
The only nontrivial computation is for $\tau_{14}$. By
\eqref{eqcur1}, we have
\begin{equation}
\tau_{14}=g_{i\bar i j\bar j k\bar k l\bar l}=-\partial_{k\bar k
l\bar l}R_{i\bar i j\bar j}+\partial_{k\bar k l  \bar
l}\left(g^{m\bar n}g_{m\bar i\bar j}g_{i\bar n j}\right).
\end{equation}
So $\tau_{14}$ follows from a tedious but straightforward
computation that
\begin{align*}
-\partial_{k\bar k l\bar l}R_{i\bar i j\bar
j}&=-3\sigma_4-6\sigma_5-3\sigma_6+6\sigma_7+7\sigma_9+4\sigma_{10}+8\sigma_{12}
-\sigma_{14}-4\sigma_{15},\\
\partial_{k\bar k l  \bar l}\left(g^{m\bar
n}g_{m\bar i\bar j}g_{i\bar n
j}\right)&=-6\sigma_{5}+4\sigma_{10}+2\sigma_{12}+3\sigma_{13}-2\sigma_{15}.
\end{align*}

Substituting \eqref{eqsz} into \eqref{eqthreeing}, we can get the
coefficients in \eqref{eqthreeinr}.
\begin{gather*}
c_1=1/48,\quad c_2=-1/12,\quad
c_3=1/48,\quad c_4=-1/8,\quad c_5=0,\\
c_6=5/24,\quad c_7=-1/12,\quad c_8=1/6,\quad
c_9=-3/8,\quad c_{10}=0,\\
c_{11}=0,\quad c_{12}=-1/4,\quad c_{13}=1/24,\quad c_{14}=1/8,\quad
c_{15}=1/24.
\end{gather*}
These results agree with the computations by Engli\v{s} \cite{Eng}. Note
that our convention of curvatures $R_{i\bar j k\bar l}, R_{i\bar j},
\rho$ in Section \ref{sectensor} all differ by a minus sign with
that of \cite{Eng}.

\vskip 30pt
\section{Tables of $z(G)$ for $G\in \mathcal G(4)$} \label{apfour}

By Proposition \ref{disjoint} and Corollary \ref{conn}, we need only
list $z(G)$ for the $51$ strongly connected stable graphs in the
following three tables. Note that if $G$ has $n$ vertices with the
adjacency matrix $A$, then
\begin{equation}
|{\rm Aut}(G)|=\prod_{1\leq i,j\leq n} A_{ij}!\cdot
\#\{\varphi\in S_n\mid A_{\varphi(i),\varphi(j)}=A_{ij},\,\forall 1\leq i,j\leq n\},
\end{equation}
where $S_n$ is the symmetric group of $n$ elements.

\begin{table}[h] \footnotesize
\centering 
\begin{tabular}{|c|c|c|c|c|c|}

\hline $\xymatrix{*+[o][F-]{5}}$
     & $\xymatrix{
        *+[o][F-]{2} \ar@/^/[r]^2
         &
        \circ \ar@/^/[l]^2} $
     & $\xymatrix{ *+[o][F-]{1}  \ar@/^/[r]^3 & \circ \ar@/^/[l]^2 }$
     & $\xymatrix{ \circ  \ar@/^/[r]^4 & \circ \ar@/^/[l]^2}$
     & $\xymatrix{ *+[o][F-]{1}  \ar@/^/[r]^2 & \circ \ar@/^/[l]^3 }$
     & $\xymatrix{ \circ  \ar@/^/[r]^3 & \circ \ar@/^/[l]^3}$
\\
\hline $-1/30$ & $5/8$ & $1/2$ & $7/48$ & $1/2$ & $1/9$
\\
\hline $\xymatrix{
        *+[o][F-]{1} \ar@/^/[r]^1
         &
        *+[o][F-]{3} \ar@/^/[l]^1} $
     & $\xymatrix{
        *+[o][F-]{1} \ar@/^/[r]^1
         &
        *+[o][F-]{2} \ar@/^/[l]^2} $
     & $\xymatrix{ *+[o][F-]{1}  \ar@/^/[r]^1 & *+[o][F-]{1} \ar@/^/[l]^3}$
     & $\xymatrix{
        *+[o][F-]{1} \ar@/^/[r]^2
         &
        *+[o][F-]{2} \ar@/^/[l]^1} $
     & $\xymatrix{ *+[o][F-]{1}  \ar@/^/[r]^2 & *+[o][F-]{1}
       \ar@/^/[l]^2}$
     & $\xymatrix{
        *+[o][F-]{2} \ar@/^/[r]^1
         &
        *+[o][F-]{2} \ar@/^/[l]^1} $
\\
\hline $1/6$ & $1/2$ & $1/2$ & $1/2$ & $1/2$ & $0$
\\
\hline
\end{tabular}
\end{table}
\begin{table}[h]  \footnotesize
\centering 
\begin{tabular}{|c|c|c|c|c|}
\hline
       \begin{minipage}{0.7in}
             $\xymatrix@C=3mm@R=6mm{
                & \circ \ar@/^-0.3pc/@<0.2ex>[dl]^{1} \ar@/^0.3pc/@<0.7ex>[dr]^{2}             \\
          *+[o][F-]{1} \ar@/^0.3pc/@<0.7ex>[ur]^{1}  & &
          \circ \ar@/^-0.3pc/@<0.2ex>[ul]^{2}
         }$
         \end{minipage}

      & \begin{minipage}{0.7in}
             $\xymatrix@C=3mm@R=6mm{
                & \circ \ar@/^-0.3pc/@<0.2ex>[dl]^{2}             \\
         \circ \ar@/^0.3pc/@<0.7ex>[ur]^{1} \ar@/^-0.3pc/@<0.2ex>[rr]^{2} & &  \circ
         \ar@/^0.3pc/@<0.7ex>[ll]^{1} \ar[ul]_{1}
         }$
         \end{minipage}
      & \begin{minipage}{0.7in}
             $\xymatrix@C=3mm@R=6mm{
                & \circ \ar@/^-0.3pc/@<0.2ex>[dl]^{2}             \\
         *+[o][F-]{1} \ar@/^0.3pc/@<0.7ex>[ur]^{1} \ar[rr]_{1} & &  *+[o][F-]{1}
          \ar[ul]_{1}
         }$
         \end{minipage}
         & \begin{minipage}{0.7in}
             $\xymatrix@C=3mm@R=6mm{
                & \circ \ar@/^-0.3pc/@<0.2ex>[dl]^{2}             \\
         \circ \ar@/^0.3pc/@<0.7ex>[ur]^{1} \ar[rr]_{2} & &  *+[o][F-]{1}
          \ar[ul]_{1}
         }$
         \end{minipage}
     &\begin{minipage}{0.7in}
             $\xymatrix@C=3mm@R=6mm{
                & \circ \ar@/^-0.3pc/@<0.2ex>[dl]^{2}             \\
         \circ \ar@/^0.3pc/@<0.7ex>[ur]^{2} \ar[rr]_{1} & &  *+[o][F-]{1}
          \ar[ul]_{1}
         }$
         \end{minipage}

\\
\hline  $-1/4$ & $-7/4$ & $-1$ & $-1$ & $-1/2$
\\
\hline
      \begin{minipage}{0.7in}
             $\xymatrix@C=3mm@R=6mm{
                & \circ \ar@/^-0.3pc/@<0.2ex>[dl]^{2}            \\
         \circ \ar@/^0.3pc/@<0.7ex>[ur]^{1} \ar@/^-0.3pc/@<0.2ex>[rr]^{1} & &  *+[o][F-]{1}
         \ar@/^0.3pc/@<0.7ex>[ll]^{1} \ar[ul]_{1}
         }$
         \end{minipage}
     & \begin{minipage}{0.7in}
             $\xymatrix@C=3mm@R=6mm{
                & \circ \ar@/^-0.3pc/@<0.2ex>[dl]^{2}            \\
         \circ \ar@/^0.3pc/@<0.7ex>[ur]^{1} \ar[rr]_{1} & &  *+[o][F-]{2}
          \ar[ul]_{1}
         }$
         \end{minipage}
      & \begin{minipage}{0.7in}$\xymatrix@C=3mm@R=6mm{
                & \circ \ar[dl]_{2}             \\
         *+[o][F-]{1} \ar[rr]_{2}  & &  \circ \ar[ul]_{2}
         }$
         \end{minipage}
     &
         \begin{minipage}{0.7in}$\xymatrix@C=3mm@R=6mm{
                & \circ \ar[dl]_{2}             \\
        \circ \ar[rr]_{3}  & &  \circ \ar[ul]_{2}
         }$
         \end{minipage}
     &  \begin{minipage}{0.7in}
             $\xymatrix@C=3mm@R=6mm{
                & \circ \ar@/^-0.3pc/@<0.2ex>[dl]^{2}             \\
         \circ \ar@/^0.3pc/@<0.7ex>[ur]^{1} \ar[rr]_{2} & &  \circ
          \ar[ul]_{2}
         }$
         \end{minipage}

\\
\hline  $-3/2$  & $-1/4$ & $-1$ & $-11/24$ & $-9/8$
\\
\hline
      \begin{minipage}{0.7in}$\xymatrix@C=3mm@R=6mm{
                & \circ \ar[dl]_{2}             \\
        *+[o][F-]{1} \ar[rr]_{1}  & &  *+[o][F-]{1} \ar[ul]_{2}
         }$
         \end{minipage}
     & \begin{minipage}{0.7in}
             $\xymatrix@C=3mm@R=6mm{
                & \circ \ar@/^-0.3pc/@<0.2ex>[dl]^{2}             \\
         \circ \ar@/^0.3pc/@<0.7ex>[ur]^{1} \ar[rr]_{1} & &  *+[o][F-]{1}
          \ar[ul]_{2}
         }$
         \end{minipage}
      & \begin{minipage}{0.7in}
             $\xymatrix@C=3mm@R=6mm{
                & \circ \ar@/^-0.3pc/@<0.2ex>[dl]^{3}             \\
         \circ \ar@/^0.3pc/@<0.7ex>[ur]^{1} \ar[rr]_{1} & &  *+[o][F-]{1}
          \ar[ul]_{1}
         }$
         \end{minipage}
     &
         \begin{minipage}{0.7in}
             $\xymatrix@C=3mm@R=6mm{
                & \circ \ar@/^-0.3pc/@<0.2ex>[dl]^{1} \ar[dr]^{1}             \\
        *+[o][F-]{1} \ar@/^0.3pc/@<0.7ex>[ur]^{2}  & &  *+[o][F-]{1}
         \ar[ll]^{1}
         }$
         \end{minipage}
     &  \begin{minipage}{0.7in}
             $\xymatrix@C=3mm@R=6mm{
                & \circ \ar@/^-0.3pc/@<0.2ex>[dl]^{1} \ar[dr]^{1}             \\
         \circ \ar@/^0.3pc/@<0.7ex>[ur]^{2} \ar@/^-0.3pc/@<0.2ex>[rr]^{1} & &  *+[o][F-]{1}
         \ar@/^0.3pc/@<0.7ex>[ll]^{1}
         }$
         \end{minipage}

\\
\hline $-1$ & $-1$ & $-1/2$ & $-1$ & $-3/2$
\\
\hline
     \begin{minipage}{0.7in}
             $\xymatrix@C=3mm@R=6mm{
                & \circ \ar@/^-0.3pc/@<0.2ex>[dl]^{1} \ar@/^0.3pc/@<0.7ex>[dr]^{1}             \\
         *+[o][F-]{1} \ar@/^0.3pc/@<0.7ex>[ur]^{1} \ar@/^-0.3pc/@<0.2ex>[rr]^{1} & &  \circ
         \ar@/^0.3pc/@<0.7ex>[ll]^{1} \ar@/^-0.3pc/@<0.2ex>[ul]^{1}
         }$
         \end{minipage}
& \begin{minipage}{0.7in}
             $\xymatrix@C=3mm@R=6mm{
                & \circ \ar@/^-0.3pc/@<0.2ex>[dl]^{1} \ar@/^0.3pc/@<0.7ex>[dr]^{1}             \\
         \circ \ar@/^0.3pc/@<0.7ex>[ur]^{1} \ar@/^-0.3pc/@<0.2ex>[rr]^{2} & &  \circ
         \ar@/^0.3pc/@<0.7ex>[ll]^{1} \ar@/^-0.3pc/@<0.2ex>[ul]^{1}
         }$
         \end{minipage}

     &  \begin{minipage}{0.7in}
             $\xymatrix@C=3mm@R=6mm{
                & \circ \ar@/^-0.3pc/@<0.2ex>[dl]^{1} \ar@/^0.3pc/@<0.7ex>[dr]^{1}             \\
          *+[o][F-]{2} \ar@/^0.3pc/@<0.7ex>[ur]^{1}  & &
           *+[o][F-]{1} \ar@/^-0.3pc/@<0.2ex>[ul]^{1}
         }$
         \end{minipage}
     & \begin{minipage}{0.7in}
             $\xymatrix@C=3mm@R=6mm{
                & \circ \ar@/^-0.3pc/@<0.2ex>[dl]^{1} \ar@/^0.3pc/@<0.7ex>[dr]^{1}             \\
         *+[o][F-]{1}\ar@/^0.3pc/@<0.7ex>[ur]^{1} \ar[rr]_{1} & &  *+[o][F-]{1}
          \ar@/^-0.3pc/@<0.2ex>[ul]^{1}
         }$
         \end{minipage}
     & \begin{minipage}{0.7in}
             $\xymatrix@C=3mm@R=6mm{
                & \circ \ar@/^-0.3pc/@<0.2ex>[dl]^{1} \ar@/^0.3pc/@<0.7ex>[dr]^{1}             \\
         *+[o][F-]{1} \ar@/^0.3pc/@<0.7ex>[ur]^{2}  & &  *+[o][F-]{1}
         \ar@/^-0.3pc/@<0.2ex>[ul]^{1}
         }$
         \end{minipage}

\\
\hline  $-2$ & $-3$ & $1/2$ & $-1$ & $0$
\\
\hline
     \begin{minipage}{0.7in}
             $\xymatrix@C=3mm@R=6mm{
                & \circ \ar@/^-0.3pc/@<0.2ex>[dl]^{2} \ar@/^0.3pc/@<0.7ex>[dr]^{1}             \\
         *+[o][F-]{1} \ar@/^0.3pc/@<0.7ex>[ur]^{1}  & &  *+[o][F-]{1}
         \ar@/^-0.3pc/@<0.2ex>[ul]^{1}
         }$
         \end{minipage}
   & \begin{minipage}{0.7in}
             $\xymatrix@C=3mm@R=6mm{
                & *+[o][F-]{1} \ar@/^-0.3pc/@<0.2ex>[dl]^{1} \ar@/^0.3pc/@<0.7ex>[dr]^{1}             \\
         *+[o][F-]{1} \ar@/^0.3pc/@<0.7ex>[ur]^{1}  & &  *+[o][F-]{1}
         \ar@/^-0.3pc/@<0.2ex>[ul]^{1}
         }$
         \end{minipage}

     &  \begin{minipage}{0.7in}$\xymatrix@C=3mm@R=6mm{
                & *+[o][F-]{1} \ar[dl]_{1}             \\
        *+[o][F-]{2} \ar[rr]_{1}  & &  *+[o][F-]{1} \ar[ul]_{1}
         }$
         \end{minipage}
     & \begin{minipage}{0.7in}$\xymatrix@C=3mm@R=6mm{
                & *+[o][F-]{1} \ar[dl]_{1}             \\
        *+[o][F-]{1} \ar[rr]_{2}  & &  *+[o][F-]{1} \ar[ul]_{1}
         }$
         \end{minipage}
     & \begin{minipage}{0.7in}
             $\xymatrix@C=3mm@R=6mm{
                & *+[o][F-]{1} \ar@/^-0.3pc/@<0.2ex>[dl]^{1}              \\
         *+[o][F-]{1}\ar@/^0.3pc/@<0.7ex>[ur]^{1} \ar[rr]_{1} & &  *+[o][F-]{1}
          \ar[ul]_{1}
         }$
         \end{minipage}

\\
\hline  $0$ & $0$ & $-1/2$ & $-1$ &$-1$
\\
\hline
\end{tabular}
\end{table}


\begin{table}[h] \footnotesize
\centering 
\begin{tabular}{|c|c|c|c|c|}

\hline
     \begin{minipage}{0.7in} \xymatrix{
            \circ  \ar@/^0.3pc/[r]^2
                & \circ \ar@/^0.3pc/[l]^1 \ar[d]^{1}  \\
            \circ  \ar@/^0.3pc/[r]^1 \ar[u]^{1}
                & \circ \ar@/^0.3pc/[l]^2            }
                \end{minipage}
     & \begin{minipage}{0.7in} \xymatrix{
  \circ  \ar[r]^{2}
                & \circ \ar[d]^{2}  \\
  \circ \ar[u]^{2}
                & \circ  \ar[l]_{2}           }
                \end{minipage}
     & \begin{minipage}{0.7in} \xymatrix{
            \circ  \ar[d]_{1}\ar@/^0.3pc/[r]^1
                & \circ \ar@/^0.3pc/[l]^2   \\
            *+[o][F-]{1} \ar[r]_1
                & *+[o][F-]{1}    \ar[u]_{1}      }
                \end{minipage}
     &  \begin{minipage}{0.7in}\xymatrix{
  \circ \ar@/^0.3pc/[d]^{2}
                & \circ \ar[l]_{1} \ar@/^0.3pc/[d]^{1}  \\
  \circ \ar@/^0.3pc/[u]^{1} \ar@{>}[ur]_{1}
                & *+[o][F-]{1} \ar@/^0.3pc/[u]^{1}            }
                \end{minipage}
     & \begin{minipage}{0.7in}\xymatrix{
  \circ \ar[d]_{2}
                & \circ  \ar[l]_{2} \\
  \circ \ar@{>}[ur]^{1} \ar[r]_{1}
                & *+[o][F-]{1}   \ar[u]_{1}         }
                \end{minipage}
\\
\hline  $3/8$ & $15/64$ & $1$ & $-1/2$ & $1$
\\

\hline
      \begin{minipage}{0.7in}
\xymatrix{
  \circ \ar[d]_{1}\ar@/^0.6pc/[dr]|-{1}
                & \circ  \ar[l]_{2} \\
  \circ \ar@/^0.6pc/[ur]|-{1} \ar@/^0.3pc/[r]^{1}
                & \circ \ar@/^0.3pc/[l]^{1}   \ar[u]_{1}     }
\end{minipage}
     & \begin{minipage}{0.7in}
\xymatrix{
  *+[o][F-]{1} \ar[r]^{1}
                & \circ   \ar@{>}[dl]_{2}\\
  \circ  \ar[r]_{1}  \ar[u]^{1}
                & *+[o][F-]{1}    \ar[u]_{1}          }
\end{minipage}
     & \begin{minipage}{0.7in} \xymatrix{
            \circ  \ar@/^0.3pc/[r]^1 \ar@/^0.3pc/[d]^{1}
                & \circ \ar@/^0.3pc/[l]^1 \ar@/^0.3pc/[d]^{1}  \\
            \circ  \ar@/^0.3pc/[r]^1 \ar@/^0.3pc/[u]^{1}
                & \circ \ar@/^0.3pc/[l]^1 \ar@/^0.3pc/[u]^{1}           }
                \end{minipage}
     & \begin{minipage}{0.7in} \xymatrix{
            \circ  \ar@/^0.3pc/[r]^1 \ar@/^0.3pc/[d]^{1}
                & \circ \ar@/^0.3pc/[l]^1 \ar@/^0.3pc/[d]^{1}  \\
            *+[o][F-]{1}  \ar@/^0.3pc/[u]^{1}
                & *+[o][F-]{1} \ar@/^0.3pc/[u]^{1}           }
                \end{minipage}
     &  \begin{minipage}{0.7in}
\xymatrix{
  \circ \ar@/^0.3pc/[d]^{1} \ar@/^0.3pc/[r]^{1}
                & \circ \ar[d]^{1} \ar@/^0.3pc/[l]^1 \\
  \circ \ar@/_0.3pc/[ur]_{1} \ar@/^0.3pc/[u]^{1}
                & *+[o][F-]{1}    \ar[l]^{1}         }
 \end{minipage}
\\
\hline  $2$ & $0$ & $3/8$ & $-1/2$ & $2$
\\

\hline
     \begin{minipage}{0.7in}
\xymatrix{
  \circ \ar@/^-0.3pc/@<-0.1ex>[d]^{1} \ar@/_0.6pc/[dr]|-{1}
                & \circ  \ar[l]_{1} \ar@/^0.3pc/[d]^{1}\\
  \circ \ar@/^0.6pc/[ur]|-{1} \ar@/^0.3pc/@<0.9ex>[u]^{1}
                & \circ  \ar@/^0.3pc/[u]^{1} \ar[l]^{1}          }
\end{minipage}
     & \begin{minipage}{0.7in}
\xymatrix{
  *+[o][F-]{1}\ar[d]_{1}
                & \circ  \ar[l]_{1} \ar@/^0.3pc/@{>}[dl]^{1}\\
  \circ \ar@/^0.3pc/@{>}[ur]^{1} \ar[r]_{1}
                & *+[o][F-]{1}   \ar[u]_{1}         }
\end{minipage}
     & \begin{minipage}{0.7in}
\xymatrix{
  \circ \ar@/^0.3pc/[d]^{1} \ar@{>}[dr]^{1}
                & *+[o][F-]{1} \ar[l]_1 \\
 *+[o][F-]{1} \ar@/^0.3pc/[u]^{1}
                & *+[o][F-]{1}     \ar[u]_{1}       }
 \end{minipage}
     & \begin{minipage}{0.7in}\xymatrix{
  *+[o][F-]{1} \ar[d]_{1}
                & *+[o][F-]{1}  \ar[l]_{1} \\
  *+[o][F-]{1}  \ar[r]_{1}
                & *+[o][F-]{1}   \ar[u]_{1}         }
                \end{minipage}
     &
\\
\hline  $5/4$ & $1/2$ & $0$ & $1/4$ &
\\
\hline
\end{tabular}
\end{table}

$$ \ \ \ \ $$

\

\end{document}